\documentclass[11pt,reqno]{amsart}
\usepackage{tikz}
\usepackage{subfig}
\usepackage{epstopdf}
\usepackage{epsfig}
\usepackage{math}
\graphicspath{{./figures/}}
\usepackage{tikz}
\usetikzlibrary{shapes,arrows}
\tikzstyle{decision} = [diamond, draw, fill=blue!20, 
    text width=4.5em, text badly centered, node distance=3cm, inner sep=0pt]
\tikzstyle{block} = [rectangle, draw, fill=blue!20, 
    text width=5em, text centered, rounded corners, minimum height=4em]
\tikzstyle{line} = [draw, -latex']
\tikzstyle{cloud} = [draw, ellipse,fill=red!20, node distance=3cm,
    minimum height=2em]

\usetikzlibrary{positioning}
\tikzset{main node/.style={circle,fill=blue!20,draw,minimum size=1cm,inner sep=0pt},  }

\begin{document}
\title[Discrete Schr{\"o}dinger equations on graphs]{A discrete Schr{\"o}dinger equation via optimal transport on graphs}
\author[Chow]{Shui-Nee Chow}
\address{School of Mathematics, Georgia Institute of Technology,
Atlanta, GA 30332 U.S.A.}
\email{chow@math.gatech.edu}
\author[Li]{Wuchen Li}
\email{wcli@math.ucla.edu}
\address{Mathematics department, UCLA,
Los Angeles, CA 90095 U.S.A.}

\author[Zhou]{Haomin Zhou}
\email{haomin.zhou@math.gatech.edu}

\thanks{This work is partially supported by NSF 
Awards DMS\textendash{}1419027,  DMS-1620345, 
and ONR Award N000141310408.}

\keywords{Nonlinear Schr{\"o}dinger equations; Optimal transport; Fisher information; Nelson's approach.}
\begin{abstract}
In 1966, Edward Nelson presented an interesting derivation of the Schr{\"o}dinger equation using Brownian motion. Recently, this derivation is linked to the theory of optimal transport, which shows that the Schr{\"o}dinger equation is a Hamiltonian system on the probability density manifold equipped with the Wasserstein metric. In this paper, we consider similar matters on a finite graph. By using discrete optimal transport and its corresponding Nelson's approach, we derive a discrete Schr{\"o}dinger equation on a finite graph. The proposed system is quite different from the commonly referred discretized Schr{\"o}dinger equations. It is a system of nonlinear ordinary differential equations (ODEs) with many desirable properties. Several numerical examples are presented to illustrate the properties.  
\end{abstract}
\maketitle 

\section{Introduction}
 The nonlinear Schr{\"o}dinger equation (NLS) given in the form of 
\begin{equation}\label{NLS}
hi \frac{\partial}{\partial t}\Psi(t,x)=-\frac{h^2}{2}\Delta \Psi(t,x)+\Psi(t,x) \mathbb{V}(x)+ \Psi(t,x)\int_{\mathbb{R}^d}\mathbb{W}(x,y)|\Psi(t,y)|^2dy  \ ,
\end{equation}
plays vital roles in many areas in physical sciences \cite{Bao1, SE}.
The unknown $\Psi(t,x)$ is a complex wave function for  $x\in\mathbb{R}^d$, $h>0$ is the Planck constant, and $\mathbb{V}(x)$ and $\mathbb{W}(x,y)$ are real valued functions, referred as linear and interaction potentials respectively. Although the NLS provides accurate predictions to various physical phenomena, its formulation is very different from the classical mechanics, and it cannot be easily interpreted by the Newton's law.

To bridge the difference, Edward Nelson provided an compelling approach in 1966 \cite{Nelson1}. 
He derived the NLS by means in classical mechanics in conjunction with variational principles and stochastic diffusion processes \cite{C1, L1}. To better understand his idea, we recall that the NLS \eqref{NLS} has a fluid dynamics formulation, named Madelung system \cite{Madelung}.
By introducing a change of variables,
 $\Psi(t,x)=\sqrt{ \rho(t,x)}e^{i S(t,x)/h}$, %\frac{ S(t,x)}{h}}$, 
one can rewrite \eqref{NLS} as
\begin{equation}\label{pairs}
 \left \{
 \begin{aligned}
 &\frac{\partial  \rho}{\partial t}+\nabla \cdot ( \rho \nabla  S)=0\ ;\\
 &\frac{\partial  S}{\partial t}+\frac{1}{2}(\nabla  S)^2+\frac{h^2}{8}\frac{\delta}{\delta  \rho(x)}\mathcal{I}( \rho)+\mathbb{V}(x)+\int_{\mathbb{R}^d}\mathbb{W}(x,y)\rho(t,y)dy=0\ ,
 \end{aligned}
 \right.
 \end{equation} 
where $\rho$, $S$ are unknown real valued functions, $\frac{\delta}{\delta\rho(x)}$ is the $L^2$ first variation operator, and $\mathcal{I}( \rho)=\int_{\mathbb{R}^d}(\nabla\log \rho(x))^2 \rho(x)dx$ is the Fisher information \cite{Fisher}. 
%A common known formulation of $\frac{\delta}{\delta\rho(x)}\mathcal{I}(\rho)=-\frac{\Delta\sqrt{\rho}}{\sqrt{\rho}}$. 
Nelson constructed a Lagrangian in the space of probability density functions, and then used calculus of variation to derive \eqref{pairs}. We shall give a brief review on his approach in Section 3. Readers can find more details in \cite{Nelson2}. Recently, Nelson's approach is linked to the framework of optimal transport theory \cite{Eric1, vil2008}, which has been developed in the past few decades \cite{Gangbo2, bb, Gangbo, vil2008}. The theory shows that the probability density space equipped with the optimal transport distance, also known as the Wasserstein metric, becomes a Riemannian manifold, and the NLS is a Hamiltonian system on this density manifold \cite{L1, v2008}. 

In this paper, we consider similar matters in discrete spaces, such as finite graphs. There are two reasons motivating us to conduct this investigation. On one side, Nelson's derivation is based on a variational principle, which makes his approach more attractive. However, although different formulations of the NLS on graphs have been introduced in physics and mathematics \cite{SE2, SE1, degond, SE3, NLSG}, not much is known through Nelson's approach, mainly  because the theory of discrete optimal transport has not been seriously explored until the past few years \cite{chow2012, EM1, M}.  On the other side, most of the discrete formulations for the NLS, especially those defined on lattices, are obtained by discretizations of the continuous NLS. Some important properties, such as conservation of energy, or dispersion relation, or time transverse (gauge) invariant, can be lost due to the discretizations. As reported in a recent survey on numerical methods for the NLS \cite{Bao1}, none of the commonly used schemes has all those features simultaneously. We would like to examine whether Nelson's approach can provide a systematic strategy for constructing the discrete NLS on general graphs in order to retain those desirable properties. Our investigation confirms this assertion.  

We follow the settings given in \cite{li-theory, li-thesis} to derive the discrete NLS, which turns out to be a system of ordinary differential equations. The derivation utilizes the optimal transport distance and the Fokker-Planck equation on a graph. The main results are sketched here. 

Consider a weighted finite graph $G=(V, E, \omega)$, where 
 $V$ %=\{1,2,\cdots, n\}$ 
is the vertex set, $E$ the edge set, and $\omega_{jl} \in \omega$ the weight of edge $(j,l)\in E$ satisfying $\omega_{lj}=\omega_{jl}>0$. We assume that $G$ is undirected, contains no self loops or multiple edges. Given a linear potential $\mathbb{V}_j$ on each note $j$ and an interactive potential $\mathbb{W}_{jl}$, with $\mathbb{W}_{lj}=\mathbb{W}_{jl}$, for any two nodes $(j,l)\in E$. Nelson's approach leads to the following ODEs: \begin{equation}\label{GNLS}
 \left \{
 \begin{aligned}
 &\frac{d\rho_j}{dt}+\sum_{l\in N(j)}\omega_{jl}( S_j- S_l)g_{jl}(\rho)=0\ ;\\
 &\frac{d S_j}{dt}+\frac{1}{2}\sum_{l\in N(j)}\omega_{jl}( S_j- S_l)^2\frac{\partial g_{jl}}{\partial \rho_j}+\frac{h^2}{8}\frac{\partial }{\partial \rho_j}\mathcal{I}(\rho)+\mathbb{V}_j+\sum_{l=1}^n \mathbb{W}_{jl}\rho_l=0\ ,
 \end{aligned}
 \right.
 \end{equation}
where $\rho_j(t)$ and $S_j(t)$ are the probability density and potential function at time $t$ on node $j$ respectively, $N(j)=\{l\in V\colon (j,l)\in E\}$ is the adjacency set of node $j$, $g_{jl}(\rho)=\frac{\rho_j+\rho_l}{2}$ represents the weight of probability density on the edge $(j,l)\in E$, and
\begin{equation*}
 \mathcal{I}(\rho):=\frac{1}{2}\sum_{(j,l)\in E}\omega_{jl}(\log\rho_j-\log\rho_l)^2g_{jl}(\rho)\ 
 \end{equation*}
 denotes the discrete Fisher information. Using $\rho_j(t)$ and $S_j(t)$, we reconstruct a complex wave equation, for the discrete NLS, on the graph. The derived discrete NLS is very different from the commonly seen ones, with the most notable distinction being a nonlinear graph Laplacian, which has not been reported before.

We shall prove that the initial value ODE \eqref{GNLS} is well defined with several favorable properties. For example, it is a Hamiltonian system that conserves the total mass and total energy. It is time reversible and gauge invariant. Its stationary solution is related to the discrete ground state, whose formulation has many desirable properties similar to those in the continuous space. In addition, a Hamiltonian matrix is introduced to study the stability of ground states. This Hamiltonian matrix is a symplectic decomposition of two nonlinear graph Laplacian matrices. One is from discrete optimal transport geometry, and the other is induced by the Hessian matrix of discrete Fisher information.
  
Our paper is arranged as follows. In Section 2, we discuss the necessity of using nonlinear Laplacian on graphs. In Section 3, we briefly review Nelson's approach, and then followed by the derivation of \eqref{GNLS} in Section 4. We show several interesting dynamical properties of \eqref{GNLS} in Sections 5 and 6. Several numerical examples are provided in Section 7. 

\section{Why nonlinear Laplacian on graphs?}
To answer the question, we consider the classical linear  Schr{\"o}dinger equation without potentials, 
\begin{equation} \label{lsq}
i\frac{\partial }{\partial t}\Psi(t,x)=-\frac{1}{2}\Delta \Psi(t,x)\ ,\quad x\in \mathbb{R}^d.
\end{equation}
It is well known that \eqref{lsq} admits plane wave solutions given in the form of $$\Psi(x,t)=Ae^{i(k\cdot x-\mu t)}\ ,$$
as long as the time frequency $\mu$ and the spatial wave number $k$ satisfying the so called {\em dispersion relation}: 
$$\mu=\frac{|k|^2}{2}\ .$$

Such a simple property may become problematic on graphs. To illstruate the challenges we face, let us consider the regular lattice in $\mathbb{R}^d$ for its simplicity, or even periodic lattices if one wishes to avoid dealing with the boundaries. In the lattice, we assume that every node has the same number of adjacent nodes. The weight on each edge is uniformly given by $\Delta x$, and the coordinate value for node $j$ is $x_j = j \Delta x$. 

%We shall show that the dispersion relation cannot be held for any linear discretization of \eqref{lsq}. To be precise, let us consider to discretize \eqref{lsq} on a uniform lattice in $\mathbb{R}^d$, with the stepsize $\Delta x>0$ and $x_j=j\Delta x$. 
On the lattice, any linear spatial discretization of \eqref{lsq} can be expressed as
\begin{equation} \label{llap}
i\frac{d\Psi_j}{dt}=-\frac{1}{2}\sum_{l\in N(j)}C_{jl}\Psi_l\ ,
\end{equation}
where $\{C_{jl}\}'s$ are selected, not all zeros, constants used to approximate the Laplace operator in \eqref{lsq}.
Assume that the discrete plane wave $\Psi_j(t)=Ae^{i(k \cdot x_j-\mu t)}$ satisfies \eqref{llap}, we must have 
$$ 
\mu e^{i(k \cdot j\Delta x -\mu t)   }=-\frac{1}{2}\sum_{l\in N(j) }C_{jl} e^{i(k \cdot l \Delta x-\mu t)}\ ,
$$
which is equivalent to 
$$\mu =-\frac{1}{2}\sum_{l \in N(j)}C_{jl} e^{i k \cdot (l-j)\Delta x}\ .$$

If $\mu$ and $k$ satisfy the dispersion relation $\mu = |k|^2/2$, one gets 
$$
|k|^2/2 = -\frac{1}{2}\sum_{l \in N(j) }C_{jl} e^{i k \cdot (l-j)\Delta x}\ .
$$
The left hand side is a quadratic function on the wave number $k$, while the right hand side is a periodic function consisting of a finite terms of trignometric polynomials. This implies that there are at most a finite number of values for $k$ satisfying this relation. Therefore only a finite number of 
pairs $(\mu, k)$ can form the plane wave solutions for \eqref{llap}. In contrast, 
any dispersion relation satisfying pair $(\mu, k)$ gives a plane wave solution for \eqref{lsq}, and there are infinitely many of them. We summarize this observation in the following theorem.

\begin{theorem}\label{th1}
For any linear spatial discretization of \eqref{lsq}, there are at most a finite number of pairs $(\mu, k)$ satisfying $\mu = |k|^2/2$ that can form its discrete plane wave solutions. 
%The discrete plane wave $\Psi_j(t)=e^{i(kx_j-\omega t)}$ with the dispersion relation $\omega=\frac{k^2}{2}$, does not satisfy any linear spatial discretization of \eqref{NLS}, 
%\begin{equation*} 
%i\frac{d\Psi_j}{dt}=-\frac{1}{2}\sum_{l}C_{jl}\Psi_l\ ,
%\end{equation*}
%where $C_{jl}$ are some constants that are not all zeros.  
\end{theorem}
%\begin{proof}
%Substitute the discrete plane wave into the linear discretized equation 
%$$ 
%\omega e^{i(k j\Delta x -\omega t)   }=-\frac{1}{2}\sum_{j}C_{jl} e^{i(k l \Delta x-\omega t)}\ .
%$$
%If the dispersion relation is true, then
%$$\omega=\frac{dk^2}{2}=-\frac{1}{2}\sum_{l}C_{jl} e^{i k (l-j)\Delta x}\ .$$
%This is impossible, since L.H.S. is the quadratic function of $k$, while R.H.S. is a periodic function of $k$ with the period $\frac{2\pi}{k\Delta x}$.
%\end{proof}
%Theorem \ref{th1} tells us that there is no linear interpolation of NLS satisfying the dispersion relation. Following Nelson's approach, we introduce a particular nonlinear discretization for NLS. 

An implication of this theorem is that one cannot expect every pair $(\mu, k)$, $\mu = |k|^2/2$, to give a plane wave solution for a given linear spatial discretization. In fact, only a finite number of pairs, which is a measure zero set, can do that. It also suggests that a nonlinear Laplacian on graphs must be used if one wants to construct a spatial discretization scheme that allows any pair $(\mu, k)$ to form a plane wave solution. On the other hand, this theorem does not imply that for any given pair $(\mu, k)$, one cannot find a linear Laplacian on the graph and use the pair to construct a plane wave solution for \eqref{llap}. However, if a different pair is given, one may have to switch to a different linear Laplacian on the graph.

We also note that nonlinear Laplace operators have been used, and proved to be necessary in the study of 
Fokker-Planck equations on graphs \cite{chow2012}. These observations motivated us to consider Nelson's approach to systematically construct schemes for \eqref{NLS}. 

\section{Review of Nelson's approach}
In this section, we briefly review Nelson's approach \cite{Nelson1, Nelson2}, and explain its connection with optimal transport \cite{v2008, vil2008}. To simplify the presentation, we do not consider the interactive potential 
$\mathbb{W}(x,y)$ in this section, even though such a consideration can be obtained in a straightforward manner.

Consider the following stochastic variational problem:
\begin{equation}\label{1}
\inf_{b} \{\int_0^1\mathbb{E}[\frac{1}{2}\dot X_t ^2-\mathbb{V}(X_t)] dt~:~\dot X_t=b(t,X_t)+\sqrt{h}\dot B_t\ ,~X(0)\sim  \rho^0\ ,~X(1)\sim  \rho^1\}\ ,
\end{equation}
where $b(t,x)\in\mathbb{R}^d$ can be any smooth vector field, $X_t$ is a stochastic process with prescribed probability densities $\rho^0$ and $\rho^1$ at time $0$ and $1$ respectively, $h>0$ represents the noise level, $B_t$ is a standard Brownian motion in $\mathbb{R}^d$ and $\mathbb{E}$ the expectation operator. Under suitable conditions given in \cite{Nelson1}, Nelson showed that \eqref{1} is equivalent to   
\begin{equation} \label{1.4}
\begin{split}
\inf_{b} \{\int_0^1\mathbb{E}[\frac{1}{2}(b(t, X_t)^2+h\nabla\cdot b(t, X_t))-\mathbb{V}(X_t)] dt~:~\dot X_t=b(t,X_t)+\sqrt{h}\dot B_t\ ,\\ ~X(0)\sim  \rho^0\ ,~X(1)\sim  \rho^1\}\ .
\end{split}
\end{equation}

By using the probability density function $\rho(t,x)$ defined as  
\begin{equation*}
\int_{A} \rho(t,x)dx=\textrm{Pr}(X_t\in A)\ , \quad \textrm{for any measurable set $A$\ ,}
\end{equation*}
problem \eqref{1.4} is transferred into a deterministic variational problem,
\begin{equation}\label{1.5}
\inf_{b} \int_0^1\int_{\mathbb{R}^d}[\frac{1}{2}(b^2 \rho+h\nabla\cdot b)\rho-\mathbb{V}(x) \rho ]dxdt\ ,
\end{equation}
%where the infimum is taken over all smooth vector fields $b(t,x)\in \mathbb{R}^d$, such that 
in which the density function $\rho(t,x)$ evolves according to 
%$X_t$'s density transition equation, i.e. 
the Fokker-Planck equation
\begin{equation*}
\frac{\partial \rho}{\partial t}+\nabla \cdot ( b \rho)=\frac{h}{2}\Delta  \rho\ , \quad\rho(0,\cdot)=\rho^0(\cdot)\ , \quad \rho(1,\cdot)=\rho^1(\cdot)\ .
\end{equation*}
%This is a typical optimal control problem. 

%Nelson discovered that \eqref{1.5} can be connected to the Schr{\"o}dinger equation. His discovery is based on an observation: 
%\begin{equation*}
%\nabla  \rho= \rho \nabla \log \rho\ . \end{equation*}
Noticing the facts $\nabla  \rho= \rho \nabla \log \rho\ $,
$\int_{\mathbb{R}^d}\nabla \cdot b\rho dx=-\int_{\mathbb{R}^d}b\nabla\rho dx=\int_{\mathbb{R}^d}b\cdot \rho \nabla\log\rho dx$,
and $\Delta\rho=\nabla\cdot(\nabla\rho)=\nabla\cdot(\rho\nabla\log\rho)$, then \eqref{1.5} can be rewritten into the following optimal control problem
\begin{equation}\label{2}
\inf_{b} \int_0^1 \int_{\mathbb{R}^d}[\frac{1}{2}(b^2- h b \cdot\nabla\log \rho) \rho-\mathbb{V}(x) \rho] dx dt\ ,\end{equation}
where
\begin{equation*}
\frac{\partial \rho}{\partial t}+\nabla \cdot ( \rho (b-\frac{h}{2}\nabla \log \rho))=0\ ,\quad \rho(0,\cdot)=\rho^0(\cdot)\ ,\quad \rho(1,\cdot)=\rho^1(\cdot)\ .
\end{equation*}

The {key} of Nelson's derivation is based on a {\em change of variable} given by,   
\begin{equation}\label{Nelson}
v(t,x):=b(t,x)-\frac{h}{2}\nabla \log \rho(t,x)\ .
\end{equation}
Substituting $v$ into \eqref{2}, the problem becomes
\begin{equation}\label{3}
\inf_{v}\int_0^1\{\int_{\mathbb{R}^d}\frac{1}{2}v^2 \rho dx-\frac{h^2}{8}\mathcal{I}(\rho)-\mathcal{V}(\rho) \} dt\ ,
\end{equation}
%where the infimum is taken over the smooth vector field $v(t,x)\in \mathbb{R}^d$, 
such that
\begin{equation*}
\frac{\partial  \rho}{\partial t}+\nabla \cdot ( \rho v)=0\ ,\quad \rho(0,\cdot)=\rho^0(\cdot)\ ,\quad \rho(1,\cdot)=\rho^1(\cdot)\ .
\end{equation*}
Here $\mathcal{V}(\rho):=\int_{\mathbb{R}^d}\mathbb{V}(x)\rho(x)dx$ is the linear potential energy and $\mathcal{I}(\rho):=\int_{\mathbb{R}^d}(\nabla\log\rho)^2\rho dx$ is the Fisher information. The integrant in \eqref{3} is the
Lagrangian for the optimal control problem. 

%Nelson discovered that 
%Introducing the corresponding Euler-Lagrange equation \cite{Nelson1, Nelson2}, 
The critical point (in the sense of Guerra-Morato \cite{Nelson1}) of \eqref{3} satisfies the Madelung equations \begin{equation}\label{6}
\left \{\begin{aligned}
%&v=\nabla  S\ ; \\
&\frac{\partial  \rho}{\partial t}+\nabla \cdot ( \rho \nabla  S)=0\ ;\\
&\frac{\partial  S}{\partial t}+\frac{1}{2}(\nabla  S)^2+\frac{\delta}{\delta\rho(x)}\{\frac{h^2}{8}\mathcal{I}( \rho)+\mathcal{V}(\rho)\}=0\ ,
\end{aligned}
\right.
\end{equation}
where $v(t,x)=\nabla S(t,x)$ is the optimal volecity field. The first equation in \eqref{6} is known as the continuity equation, while the second one is called Hamilton-Jacobi equation in the literature. 

Introducing a complex wave function
\begin{equation*}%\label{bohn}
\Psi(t,x)=\sqrt{ \rho(t,x)}e^{\frac{i S(t,x)}{h}}\ ,
\end{equation*}
then $\Psi(t,x)$ satisfies the linear Schr{\"o}dinger equation \begin{equation*}
hi\frac{\partial}{\partial t}\Psi=-\frac{h^2}{2}\Delta \Psi+\Psi \mathbb{V}(x)\ .%+\Psi f(|\Psi|^2).
\end{equation*}
%One can extend the above derivation to include more general interactive energies. 

Compared to the optimal transport distance, which can be defined by the well known Benamou-Brenier formula \cite{bb},
\begin{equation*}
\inf_{v}\{\int_0^1\int_{\mathbb{R}^d}v^2 \rho dxdt~ :~\frac{\partial \rho}{\partial t}+\nabla\cdot (\rho v)=0\ , \quad \textrm{$\rho(0,\cdot) = \rho^0(\cdot)\ , \quad \rho(1,\cdot) = \rho^1(\cdot)$}\}\ .
\end{equation*}
Nelson's approach can be viewed as a modified optimal transport problem in which the negative of a potential energy and the Fisher information are amended to create the Lagrangian in formulation \eqref{3}.

\section{A Schr{\"o}dinger equation on a finite graph}
Following Nelson's approach, we shall derive the discrete NLS \eqref{GNLS} on a graph via discrete optimal transport. 

%In this section, we derive the NLS on a graph \eqref{GNLS} via discrete optimal transport. 

To do so, we first review some basics of the discrete optimal transport theory developed in recent years. Consider a finite graph $G=(V, E, \omega)$. The probability set (simplex) supported on all vertices of $G$  is defined by 
\begin{equation*}
\mathcal{P}(G)=\{(\rho_j)_{j=1}^n\mid \sum_{j=1}^n\rho_j=1\ ,\quad \rho_j\geq 0\ , \quad \textrm{for any $j\in V$}\}\ ,
\end{equation*}
where $\rho_j$ is the discrete probability function at node $j$.  The interior of $\mathcal{P}(G)$ is denoted by
 $\mathcal{P}_o(G)$. 
 
Following \cite{li-theory}, we introduce some notations and operators on $G$ and $\mathcal{P}(G)$. A {\em vector field} $v$ on $G$ refers to a  
{\em skew-symmetric matrix} on the edges set $E$:  
\begin{equation*}
v:=(v_{jl})_{(j,l)\in E}\ ,\quad\textrm{with} \quad v_{lj}=-v_{jl}\ .
\end{equation*}
Given a function $S=(S_j)_{j=1}^n$ on $V$, it induces a {\em potential vector field} $\nabla_G S$ on $G$ as 
\begin{equation*}
\nabla_GS:=(\sqrt{\omega_{jl}}(S_j-S_l))_{(j,l)\in E}\ .
\end{equation*} 
For a probability function $\rho \in \mathcal{P}(G)$ and a vector field $v$, define 
the product $\rho v$, called {\em flux function} on $G$, by
%{\em  
\begin{equation*}
\rho v:=(v_{jl} g_{jl}(\rho))_{(j,l)\in E}\ ,
\end{equation*}
where $g_{jl}(\rho)$ is a chosen function on the edge 
 \begin{equation*}
 g_{jl}(\rho)=\frac{\rho_j+\rho_l}{2}\ ,\quad \textrm{for any $(j,l)\in E$}\ .
\end{equation*} 
We remark that $g_{jl}$ may have other choices, such as the logarithmic mean used in \cite{EM1}. The {\em divergence} of flux function $\rho v$ on $G$ is defined by
\begin{equation*}
\textrm{div}_G(\rho v):=-\biggl(\sum_{l\in N(j)}\sqrt{\omega_{jl}}v_{jl}g_{jl}(\rho)\biggr)_{j=1}^n\ .
\end{equation*}
Given two vector fields $v=(v_{jl})_{(j,l)\in E}$, $u=(u_{jl})_{(j,l)\in E}$ on a graph and $\rho \in \mathcal{P}(G)$,
the discrete {\em inner product} is defined by,
\begin{equation*}
(v, u)_ {\rho}:=\frac{1}{2}\sum_{(j,l)\in E} v_{jl}u_{jl}g_{jl}(\rho)\ , 
\end{equation*}
where the coefficient $1/2$ accounts for the fact that every edge in $G$ is counted twice, i.e. $(j,l), (l,j)\in E$. 

Using the notations, the Wasserstein metric on the graph can be defined by the discrete Benamou-Brenier formula \cite{bb},
\begin{definition}\label{W2}
 For any $ \rho^0$, $ \rho^1\in \mathcal{P}_o(G)$, define a metric
\begin{equation*}\label{metric}
{W}( \rho^0, \rho^1):=\inf_{v}~\{\left(\int_0^1(v, v)_ \rho dt\right)^{\frac{1}{2}}~:~
\frac{d \rho}{dt}+\textrm{div}_G(\rho v)=0\ ,\quad \rho(0)= \rho^0\ ,\quad  \rho(1)= \rho^1\}\ ,
\end{equation*}
where the infimum is taken over all vector fields $v$ on a graph, and $ \rho$ is a continuous differentiable curve
$ \rho:[0,1]\rightarrow \mathcal{P}_o(G)$. 
\end{definition}
$\mathcal{P}_o(G)$ equipped with the metric ${W}$ is a Riemannian manifold \cite{chow2012,li-theory}. 

\subsection{Nelson's approach on a finite graph}
Now, we are ready to derive the NLS on graph \eqref{GNLS} via discrete optimal transport. 

In the discrete case, the linear and interaction potentials refer to 
$$\mathcal{V}(\rho)=\sum_{j=1}^n \mathbb{V}_j\rho_j\ ,\quad \mathcal{W}(\rho)=\frac{1}{2}\sum_{l=1}^n\sum_{j=1}^n \mathbb{W}_{lj}\rho_l\rho_j\ , $$
respectively. We start with a discrete analog of Nelson's problem presented in \eqref{1.5},
\begin{equation}\label{2_new}
\inf_{b} \int_0^1 \frac{1}{2}[(b, b)_ {\rho}-h(b, \nabla_G\log \rho)_ {\rho}]-\mathcal{V}( \rho)-\mathcal{W}( \rho)dt\ ,
\end{equation}
where the infimum is taken over all discrete vector fields $b$, $\rho(t)$ satisfies the discrete Fokker-Planck equation \cite{chow2012, li-theory}: 
\begin{equation*}
\frac{d \rho}{dt}+\textrm{div}_G( \rho(b-\frac{h}{2}\nabla_G\log \rho))=0\ ,
\end{equation*}
and $ \rho(0) = \rho^0$, $ \rho(1) = \rho^1$ are given in $\mathcal{P}_o(G)$. Similar to Nelson's change of variable, we define a new vector field $v=(v_{lj})_{(l,j)\in E}$ on the graph
\begin{equation*}
v:=b-\frac{h}{2}\nabla_G \log \rho\ .
\end{equation*}
Substituting $v$ into \eqref{2_new}, we can write the objective functional as 
 \begin{equation}\label{2_Nelson}
 J(v):=\int_0^1\frac{1}{2}(v,v)_ {\rho}-{\frac{h^2}{8}\mathcal{I}(\rho)}-\mathcal{V}( \rho)-\mathcal{W}( \rho) dt\ ,
\end{equation}
where $v$ is a vector field on the graph, such that $\rho(t)\in \mathcal{P}_o(G)$ satisfies
\begin{equation*}
\frac{d \rho}{d t}+\textrm{div}_G ( \rho v)=0\ ,\quad \rho(0)=\rho^0\ ,\quad \rho(1)=\rho^1\ .
\end{equation*}
We note that any given feasible path $\rho(t)$ of \eqref{2_Nelson} is determined by $v(t)$, so we denote functional $J$ only in term of $v$, and call it {discrete Nelson's approach}. 

We also call $(v(t), \rho(t))$ a {\bf critical point} of \eqref{2_Nelson}  if \begin{equation}\label{critical}
\mathcal{J}(v+\delta v)-\mathcal{J}(v)=o(\delta v)\ ,\quad 
\forall \delta v(t) \in \mathcal{D},
\end{equation}
where 
\begin{equation*}
\begin{split}
\mathcal{D}=\{ \delta v\in C^{\infty}[0,1]~:~
&~\bar \rho(t) \in\mathcal{P}_o(G)~\textrm{is continuously differentiable, and}\ , \\
&\frac{d \bar \rho}{dt}+\textrm{div}_G( \bar \rho(v+\delta v))=0\ ,~\bar \rho(0)=\rho^0\ ,~\bar \rho(1)=\rho^1\}\ .
\end{split}
\end{equation*}
In the following theorem, we show that the critical point of the discrete Nelson's approach satisfies \eqref{GNLS}.
\begin{theorem}[Critical point of Nelson's approach]\label{derivation}
Assume there exists a critical point $(v(t), \rho(t))$ of \eqref{2_Nelson} in the sense of \eqref{critical}, which are smooth functions with respect to the time variable. Then $v(t)$ and $\rho(t)$ satisfy the following conditions:
 \begin{itemize}
\item[(a)] $v(t)$ is a potential vector field on the graph, i.e. there exists a function $S(t) = ( S_j(t))_{j=1}^n$ defined on the nodes, such that 
 \begin{equation*}
 v_{jl}(t)= \sqrt{\omega_{jl}}(S_j(t)- S_l(t))\ ,\quad \textrm{for all $t\in[0, 1]$ and $(j,l) \in E$}\ .
 \end{equation*} 
\item[(b)] For every $S(t)$ that induces $v(t)$, there exists a scalar function $C(t)\in \mathbb{R}$, independent of 
the nodes, such that $\rho(t)$ and $\bar S(t)$, defined by $\bar S(t) =(S_j(t)-C(t))_{j=1}^n$, satisfy \eqref{GNLS}:
\begin{equation*}
 \left \{
 \begin{aligned}
 &\frac{d\rho_j}{dt}+\sum_{l\in N(j)}\omega_{jl}( \bar S_j- \bar S_l)g_{jl}(\rho)=0\ ;\\
 &\frac{d \bar S_j}{dt}+\frac{1}{2}\sum_{l\in N(j)}\omega_{jl}( \bar S_j- \bar S_l)^2\frac{\partial g_{jl}}{\partial \rho_j}+\frac{h^2}{8}\frac{\partial }{\partial \rho_j}\mathcal{I}(\rho)+\mathbb{V}_j+\sum_{l=1}^n \mathbb{W}_{jl}\rho_l=0\ ,
 \end{aligned}
 \right.
 \end{equation*}

\end{itemize}
\end{theorem}
\begin{remark}
We can rewrite \eqref{GNLS} as
\begin{equation}\label{g_ham}
  \frac{d}{d t}\begin{pmatrix}
  \rho\\ S
 \end{pmatrix}=\mathbb{J} \begin{pmatrix}\frac{\partial}{\partial  \rho}\mathcal{H}\\ \frac{\partial}{\partial  S}\mathcal{H}\end{pmatrix}\ ,
  \end{equation}
 where $\mathcal{H}$ is the discrete total energy
\begin{equation*}
\mathcal{H}(\rho, S):=\frac{1}{2}(\nabla_G S, \nabla_G S)_\rho+\frac{h^2}{8}\mathcal{I}(\rho)+\mathcal{V}(\rho)+\mathcal{W}(\rho)\ ,
\end{equation*}
%$\textrm{Diff}=\begin{pmatrix}\frac{\partial}{\partial \rho},
 %\frac{\partial}{\partial S}\end{pmatrix}^T$ is the Euclidean gradient operator in $\mathbb{R}^{2n}$, 
 and  
 $$\mathbb{J}=\begin{pmatrix}
 0&\mathbb{I}\\
-\mathbb{I}& 0
\end{pmatrix}$$ is a symplectic matrix, with $\mathbb{I}\in \mathbb{R}^{n\times n}$ being the identity matrix.
The symplectic form \eqref{g_ham}, identical to \eqref{GNLS},  
is the discrete analog of \eqref{6}. Following the convention, we call the first equation the discrete 
continuity equation, and the second one the discrete Hamilton-Jacobi equation.
\end{remark}

The proof of Theorem 3 requires the following lemma, which can be viewed as the Hodge decomposition on a graph.
\begin{lemma}\label{lemma}
Given a vector field $v=(v_{jl})_{(j,l)\in E}$ on a graph  and a probability density function $ \rho\in\mathcal{P}_o(G)$, there exists a unique $\nabla_GS$, such that 
\begin{equation}\label{a}
v=\nabla_GS+u\ , \quad\textrm{with}\quad \textrm{div}_G( \rho u)=0\ . \end{equation}
\end{lemma}
\begin{proof}
The detailed proof can be found in \cite{EM1}. For the completeness of this paper, more importantly 
to introduce some notations that will be used later in the paper, we sketch the proof here.
We define a weighted graph Laplacian matrix $L(\rho)\in \mathbb{R}^{n\times n}$:
$$ L(\rho)=-D^T \Theta(\rho) D\ ,$$
where $D \in \mathbb{R}^{|E|\times |V|}$ is the discrete gradient matrix 
\begin{equation*} 
D_{(j,l)\in E, k\in V}=\begin{cases}
\sqrt{\omega_{jl}} & \textrm{if $j=k\ ;$}\\ 
-\sqrt{\omega_{jl}} & \textrm{if $l=k\ ;$}\\
0 & \textrm{otherwise}\ ;
\end{cases}
\end{equation*}
the transpose of $D$, denoted by $D^T$, is the discrete divergence matrix and
$\Theta\in \mathbb{R}^{|E|\times |E|}$ is the diagonal weighted matrix
\begin{equation*}
\Theta_{(j,l)\in E, (j',l')\in E}=\begin{cases}
g_{jl}(\rho) & \textrm{if $(j,l)=(j',l')\in E\ ;$}\\ 
0 & \textrm{otherwise}\ .
\end{cases}
\end{equation*}
We emphasize that $L(\rho)$ depends only on $\rho$. This is different from the commonly seen graph Laplace operator. %which are often independent of $\rho$.

We only need to show that there exists a unique gradient vector field $\nabla_G S$, such that
$$\textrm{div}_G(\rho \nabla_G S)=L(\rho) S=\textrm{div}_G(\rho v)\ .$$
Since $\rho\in \mathcal{P}_o(G)$ and the graph is connected, then $$ S^TL(\rho)S=\frac{1}{2}\sum_{(j,l)\in E}\omega_{jl}(S_j-S_l)^2g_{jl}(\rho)=0\ ,$$
this implies that value $0$ must be a simple eigenvalue of the weighted graph Laplacian matrix $L(\rho)$  with eigenvector $\{1,\cdots ,1\}$. Thus there exists a unique solution of $S$ up to constant shrift. Therefore $\nabla_G S$ is unique.
\end{proof}
%In fact, we can define a unique $S\in \mathbb{R}^n$, such that \eqref{a} holds. 
Furthermore, we can express 
$$L(\rho)=U\begin{pmatrix}
0 & & &\\
& \lambda_{sec}(L(\rho))& &\\
& & \ddots & \\
& & & {\lambda_{\max}(L(\rho))}
\end{pmatrix}U^{-1} \ ,$$
where $0<\lambda_{sec}(L(\rho))\leq\cdots\leq \lambda_{\max}(L(\rho))$ are $n$ eigenvalues of $L(\rho)$ arranged in the ascending order, and $U$ is a matrix whose columns are eigenvectors of $L(\rho)$. 
The pseudo-inverse of $L(\rho)$ is defined by
$$L(\rho)^{-1}=U\begin{pmatrix}
0 & & &\\
& \frac{1}{\lambda_{sec}(L(\rho))}& &\\
& & \ddots & \\
& & & \frac{1}{\lambda_{max}(L(\rho))}
\end{pmatrix}U^{-1} \ .$$
Thus $S=L(\rho)^{-1}\textrm{div}_G(\rho v)$.

\begin{proof}[Proof of Theorem \ref{derivation}]
(a) We prove that $v(t)=\nabla_G S(t)$. From Lemma \ref{lemma}, we have,
\begin{equation*}
v(t)=\nabla_G S(t)+u(t)\ ,\quad\textrm{with}\quad \textrm{div}_G( \rho(t) u(t))=0 \quad \textrm{for all $t\in [0,1]$} \ .
\end{equation*} 
%where $S(t)$ is uniquely determined by $v(t)$. 
We only need to prove $u(t)=0$ for $t\in[0,1]$ when $v(t)$ is a critical point. 
%We show this by using a particular perturbation. 
Consider a function $w(t)=(w_{jl}(t))_{(j,l)\in E}$ satisfying \begin{equation*}
 \textrm{div}_G( \rho w(t))=0\ , \quad \textrm{for $t\in[0, 1]$}\ .  
\end{equation*}
It is clear that $\epsilon w \in \mathcal{D}$ for any $\epsilon>0$ because $\frac{d\rho}{dt}+\textrm{div}_G(\rho (v+\epsilon w))=0$. 

Since $(v, \rho)$ is a critical solution of \eqref{2_Nelson}, we must have
 \begin{equation}\label{one}
\lim_{\epsilon \rightarrow 0}\frac{\mathcal{J}(v+\epsilon w)-\mathcal{J}(v)}{\epsilon}=0\ .
\end{equation}
Because $ \rho(t)$ keeps the same for the vector field $v+\epsilon w$, it implies
\begin{equation*}
\begin{split}
2\frac{\mathcal{J}( v+\epsilon w)-\mathcal{J}(v)}{\epsilon}=&\int_0^1 \frac{(v+\epsilon w,v+\epsilon w)_ {\rho}-(v,v)_ {\rho}}{\epsilon}dt\\
=& \int_0^1 \frac{(v, v)_ {\rho}+2\epsilon (v,w)_ {\rho}+\epsilon^2(w, w)_ {\rho}-(v, v)_ {\rho}}{\epsilon}dt\\
=& 2\int_0^1 (v, w)_ {\rho} dt+O(\epsilon)= 2\int_0^1(\nabla_G  S+u, w)_ {\rho} dt+O(\epsilon)\\
=&2\int_0^1(\nabla_G  S, w)_ {\rho}+(u, w)_ {\rho} dt+O(\epsilon)\\
=&2\int_0^1-\sum_{j=1}^n \textrm{div}_G( \rho  w)|_j S_j+(u,w)_ {\rho} dt+O(\epsilon)\\ 
=&\int_0^1\sum_{(j,l)\in E}u_{jl}(t)w_{jl}(t) g_{jl}(\rho(t)) dt+O(\epsilon)\ , \end{split}
\end{equation*}
where the last equality uses the fact $\textrm{div}_G( \rho w)=0$. From \eqref{one}, we get
\begin{equation*}
\int_0^1\sum_{(j,l)\in E}u_{jl}(t)w_{jl}(t)g_{jl}( \rho(t)) dt=0\ .
\end{equation*}
In particular, by taking $w(t)=u(t)$, we obtain
\begin{equation*}
\int_0^1\sum_{(j,l)\in E}u_{jl}(t)^2g_{jl}( \rho(t)) dt=0\ . \end{equation*}
Since $ \rho(t)\in \mathcal{P}_o(G)$, $g_{jl}(\rho(t))>0$ and $u(t)\in \mathcal{D}$, this implies $u(t)=0$ for $t\in[0,1]$, which proves (a).

(b) %We derive \eqref{GNLS} by the other particular perturbation. 
Since $\rho(t)\in \mathcal{P}_o(G)$ is continuous in $[0,1]$, then $\min_{i\in V,~t\in [0,1]}\rho_i(t)\geq c_0>0$. We consider a perturb function $ \rho^{\epsilon}(t)$ defined by:
\begin{equation*}
\rho^\epsilon(t)=\rho(t)+\epsilon \delta \rho(t)\ ,
\end{equation*}
where $\delta \rho(t)=(\delta\rho_j(t))_{j=1}^n$, $\delta\rho_j(t)\in C^\infty[0, 1]$ with $\sum_{j=1}^n\delta\rho_j=0$ and $\delta \rho(0)=\delta \rho(1)=0$. Let $\epsilon\sup_{0\leq t\leq 1}|\delta\rho(t)|<\frac{1}{2}c_0$, then $\rho^\epsilon(t)\in \mathcal{P}_o(G)$. Thus $L(\rho^{\epsilon}(t))^{-1}$ is well defined for $t\in[0,1]$, whose entries are smooth. From $S^\epsilon(t)=L(\rho^\epsilon(t))^{-1}\frac{d \rho^\epsilon}{dt}$, then $S^\epsilon(t)$ is smooth with respect to $t$ and $\epsilon$. Since $$\frac{d\rho^\epsilon}{dt}+\nabla_G(\rho^\epsilon (\nabla_GS^\epsilon-\nabla_G S+\nabla_G S))=0$$, we have $\nabla_G S^\epsilon(t)-\nabla_GS(t)\in D$. 

For the simplicity of presentation, we denote 
%\begin{equation*}
$$
\mathcal{F}( \rho):=\frac{h^2}{8}\mathcal{I}(\rho)+\mathcal{V}( \rho)+\mathcal{W}( \rho)\ .
$$
%\end{equation*}
By direct calculations, we have 
\begin{equation}\label{s0}
\begin{split}
\frac{\mathcal{J}(\nabla_G S^\epsilon)-\mathcal{J}(\nabla_G S)}{\epsilon}
=&\int_0^1 \frac{1}{2\epsilon}[(\nabla_G  S^\epsilon, \nabla_G  S^\epsilon)_{ \rho^{\epsilon}}-(\nabla_G S, \nabla_G S)_ {\rho}]dt-\int_0^1 \frac{1}{\epsilon}[\mathcal{F}( \rho^{\epsilon})-\mathcal{F}( \rho)] dt \ . \\
=&\int_0^1 \frac{1}{2\epsilon}[(\nabla_G  S^\epsilon, \nabla_G  S^\epsilon)_{ \rho^{\epsilon}}-(\nabla_G S, \nabla_G S)_ {\rho}]dt \quad \quad (\star)\\
&-\int_0^1 \sum_{j=1}^n\delta\rho_j\frac{\partial}{\partial\rho_j}\mathcal{F}(\rho)dt+O(\epsilon) \ . \\
\end{split}
\end{equation}
We need to estimate $(\star)$. Using the Taylor expansion of $S^\epsilon(t)$ with respect to $\epsilon$, $S^\epsilon(t)=S(t)+\epsilon \delta S(t)+o(\epsilon)$, 
where $\delta S(t)=\frac{d}{d\epsilon}S^\epsilon(t)|_{\epsilon=0}$, % and $\lim_{\epsilon \rightarrow 0}\frac{o(\epsilon)}{\epsilon}=0$, then
we obtain
\begin{equation}\label{15}
\begin{split}
(\star)=&\int_0^1\frac{1}{2\epsilon}[(\nabla_G  S+\epsilon\nabla_G \delta S, \nabla_G  S+\epsilon\nabla_G \delta S)_{ \rho^{\epsilon}}-(\nabla_G S, \nabla_G S)_ {\rho}]dt+O(\epsilon)\\
=&\int_0^1\frac{1}{2\epsilon}[(\nabla_G  S, \nabla_G S)_{ \rho^{\epsilon}}-(\nabla_G S, \nabla_G S)_ {\rho}+2\epsilon (\nabla_G S, \nabla_G \delta S)_{ \rho^{\epsilon}} ]dt+O(\epsilon)\\
=&\int_0^1\frac{1}{4}\sum_{j=1}^n\sum_{l\in N(j)}\omega_{jl}( S_j- S_l)^2[g_{jl}( \rho +\epsilon \delta \rho)-g_{jl}(\rho)]dt+\int_0^1(\nabla_G S, \nabla_G \delta S)_{\rho}dt+O(\epsilon)\\
=&\int_0^1\frac{1}{4}\sum_{(j,l)\in E}\omega_{jl}( S_j- S_l)^2[\frac{\partial g_{jl}}{\partial \rho_j}\delta \rho_j +\frac{\partial g_{jl}}{\partial \rho_l}\delta \rho_l]dt+\int_0^1(\nabla_G S, \nabla_G \delta S)_{\rho}dt+O(\epsilon)\\
=&\int_0^1\frac{1}{2}\sum_{j=1}^n [\delta \rho_j\sum_{l\in N(j)}\omega_{jl}( S_j- S_l)^2\frac{\partial g_{jl}}{\partial \rho_j}] dt+\int_0^1(\nabla_G S, \nabla_G \delta S)_{\rho}dt+O(\epsilon)\ .\\
\end{split}
\end{equation}

Next we find a connection between $\delta  \rho(t)$ and $\delta S(t)$. 
Notice that
$$\frac{d\rho}{dt}+\epsilon \frac{d}{dt}\delta \rho=\textrm{div}_G((\rho+\epsilon \delta \rho)\nabla_G (S+\epsilon \delta S+o(\epsilon))\ . $$
By comparing the order $\epsilon$ term, we have
\begin{equation*}
\begin{split}
\frac{d}{dt}\delta\rho_j=&\sum_{l\in N(j)}\omega_{jl}( S_l- S_j)[\frac{\partial g_{jl}}{\partial \rho_j}\delta \rho_j+\frac{\partial g_{jl}}{\partial \rho_l}\delta \rho_l]+\textrm{div}_G( \rho  \nabla_G\delta S)\ .
\end{split}
\end{equation*}
Then 
\begin{equation}\label{s3}
\begin{split}
\int_0^1 \sum_{j=1}^n S_j\frac{d}{dt}\delta \rho_j dt
=&\int_0^1 \sum_{j=1}^n\sum_{l\in N(j)} \omega_{lj}S_i( S_l- S_j)[\frac{\partial g_{lj}}{\partial \rho_j}\delta \rho_j+\frac{\partial g_{jl}}{\partial \rho_l}\delta \rho_l]+\sum_{j=1}^n S_j \textrm{div}_G(\rho\nabla_G\delta S)|_j dt\\
=&\int_0^1 (\frac{1}{2}+\frac{1}{2})\sum_{j=1}^n\delta \rho_j\sum_{l\in N(j)}\omega_{lj}( S_l- S_j)^2\frac{\partial g_{lj}}{\partial \rho_j}+(\nabla_G S, \nabla_G\delta S )_{\rho} dt \\
=&(T_1)+\int_0^1\frac{1}{2}\sum_{j=1}^n \delta \rho_j\sum_{l\in N(j)}\omega_{jl}( S_l- S_j)^2\frac{\partial g_{lj}}{\partial \rho_j}dt+O(\epsilon)\ ,
\end{split}
\end{equation}
where the last equality is from \eqref{15}. Substituting \eqref{s3} into \eqref{s0}, we have
\begin{equation*}
\begin{split}
0=&\lim_{\epsilon\rightarrow 0}\frac{\mathcal{J}(\nabla_G S^\epsilon)-\mathcal{J}(\nabla_G S)}{\epsilon}=\lim_{\epsilon \rightarrow 0}~\{T-\int_0^1 \sum_{j=1}^n\delta\rho_j\frac{\partial}{\partial\rho_j}\mathcal{F}(\rho)dt+O(\epsilon)\}\\
=& \int_0^1 \sum_{j=1}^n S_j\frac{d}{dt}\delta \rho_j dt-\frac{1}{2}\sum_{j=1}^n \delta \rho_j\sum_{l\in N(j)}\omega_{lj}( S_l- S_j)^2\frac{\partial g_{lj}}{\partial \rho_j}dt-\sum_{j=1}^n\delta\rho_j\frac{\partial}{\partial\rho_j}\mathcal{F}(\rho)dt\\
=&-\int_0^1 \sum_{j=1}^n \delta \rho_j \{\frac{d}{dt} S_j+\frac{1}{2}\sum_{l\in N(j)}\omega_{lj}( S_l- S_j)^2\frac{\partial g_{lj}}{\partial \rho_j}+\frac{\partial }{\partial \rho_j}\mathcal{F}( \rho)\}dt\ ,\end{split}
\end{equation*}
where the last equality is from integration by parts and $\delta \rho(0)=\delta \rho(1)=0$. Since $\delta\rho_j(t)$ with $\sum_{j=1}^n \delta\rho_j(t)=0$ can be any smooth function, we obtain
 \begin{equation*}
\frac{d}{dt} S_j+\frac{1}{2}\sum_{l\in N(j)}\omega_{jl}( S_j- S_l)^2\frac{\partial g_{jl}}{\partial \rho_j}+\frac{\partial }{\partial \rho_j}\mathcal{F}( \rho)=c(t)\ ,
\end{equation*}
for any smooth function $c(t)\in \mathbb{R}$. We denote $C(t)=\int_0^t c(s)ds$, then $\bar S_j(t) = S_j(t)-C(t)$, and together with $\rho(t)$, satisfy \eqref{GNLS}.
\end{proof}
\begin{corollary}
Let $(v(t), \rho(t))$ be a critical point of \eqref{2_Nelson}, and $S(t)$ a function on $G$ that induces $v(t)$, 
then $S(t))$ and $\rho(t)$ satisfy \eqref{GNLS} if and only if 
\begin{equation}\label{statement}
\sum_{j=1}^nS_j(t)\rho_j(t)=\sum_{j=1}^nS_j(0)\rho_j(0)+\int_0^t\{\frac{1}{2}(\nabla_G S, \nabla_G S)_\rho-\frac{h^2}{8}\mathcal{I}(\rho)-\mathcal{V}(\rho)-2\mathcal{W}(\rho)\} ds\ .
\end{equation}
\end{corollary}
\begin{proof}
%For any $(\rho(t), \bar S(t))$ in Theorem \ref{derivation}, $(\rho(t), \nabla \bar S(t))$ is unique. Suppose 
From the proof of Theorem \ref{derivation}, we know 
$$\frac{dS_j}{dt} +\frac{1}{2}\sum_{l\in N(j)}\omega_{jl}( S_j- S_l)^2\frac{\partial g_{jl}}{\partial \rho_j}+\frac{\partial }{\partial \rho_j}\mathcal{F}( \rho)=c(t)\ .$$ 
Then by direct calculations, we obtain
\begin{equation*}
\begin{split}
\frac{d}{dt}(\sum_{j=1}^nS_j(t)\rho_j(t))=&\sum_{j=1}^n[\frac{dS_j}{dt}\rho_j(t)+\frac{d\rho_j}{dt}S_j(t)]\\
=& \sum_{j=1}^n[-\frac{1}{2}\sum_{l\in N(j)}\omega_{jl}( S_j- S_l)^2\frac{\partial g_{jl}}{\partial \rho_j}-\frac{\partial }{\partial \rho_j}\mathcal{F}( \rho)+c(t)]\rho_j(t) + (\nabla_G S, \nabla_G S)_\rho    \\
=&\frac{1}{2}(\nabla_G S, \nabla_GS)_\rho-\sum_{j=1}^n\frac{\partial }{\partial \rho_j}\mathcal{F}( \rho)\rho_j(t)+c(t)\ .
\end{split}
\end{equation*}
We note 
\begin{equation*}
\begin{split}
\sum_{j=1}^n\frac{\partial }{\partial \rho_j}\mathcal{F}(\rho)\rho_j=&\sum_{j=1}^n\frac{\partial}{\partial\rho_j}(\frac{h^2}{8}\mathcal{I}(\rho)+\mathbb{V}^T\rho+\frac{1}{2}\rho^T\mathbb{W}\rho)\rho_j\\
=&\sum_{j=1}^n\frac{h^2}{8}\frac{\partial}{\partial\rho_j}\mathcal{I}(\rho)\rho_j+\mathbb{V}^T\rho+\frac{1}{2}\rho^T\mathbb{W}\rho+\frac{1}{2}\rho^T\mathbb{W}\rho\ ,\\
\end{split}
\end{equation*}
and
%\begin{equation}\label{eign}
%$\mathcal{I}(\rho)=\sum_{j=1}^n \frac{\partial}{\partial\rho_j}\mathcal{I}(\rho)\cdot\rho_j$, since 
%end{equation}
%Here \eqref{eign} is proved by a direct computation: 
\begin{equation*}
\begin{split}
\sum_{j=1}^n \frac{\partial}{\partial\rho_j}\mathcal{I}(\rho)\cdot\rho_j
=&\sum_{j=1}^n\sum_{l\in N(j)}\omega_{jl}(\log\rho_l-\log\rho_j)^2\frac{\partial g_{jl}}{\partial\rho_j}\rho_j+2\sum_{j=1}^n\sum_{l\in N(j)}\omega_{jl}\frac{1}{\rho_j}(\log\rho_j-\log\rho_l)\rho_j\\
=&\frac{1}{2}\sum_{l=1}^n\sum_{l\in N(j)}\omega_{jl}(\log\rho_j-\log\rho_l)^2g_{jl}(\rho)+2\sum_{(j,l)\in E}\omega_{jl}(\log\rho_j-\log\rho_l)\\
=&\mathcal{I}(\rho)\ .
\end{split}
\end{equation*}
%where the second equality is from $\frac{\partial g_{jl}}{\partial\rho_j}\cdot \rho_j+ \frac{\partial g_{jl}}{\partial\rho_l}\cdot \rho_l=g_{ij}(\rho)$. 
Combining the calculations together, we have
$$
\frac{d}{dt}(\sum_{j=1}^nS_j(t)\rho_j(t))=\frac{1}{2}(\nabla_G S, \nabla_GS)_\rho - \mathcal{F}(\rho)-\mathcal{W}(\rho)+ c(t)\ .
$$
Therefore $c(t)=0$ if and only if \eqref{statement} holds. 
\end{proof}

In fact, the construction of $S$ suggests  
$S= L(\rho)^{-1}\dot\rho$. This implies $$(\nabla_G S, \nabla_G S)_\rho=S^T L(\rho) S= \dot \rho^TL(\rho)^{-1}\cdot L(\rho) \cdot L(\rho)^{-1}\dot \rho= \dot \rho ^T L(\rho)^{-1}\dot\rho\ .$$
Thus the discrete Nelson's problem can be re-written as a geometric variational problem on the probability density manifold  $\mathcal{P}_o(G)$
\begin{equation} \label{dnp2}
\inf_{\rho}\{\int_0^1 \dot\rho^T L(\rho)^{-1}\dot \rho -\frac{h^2}{8}\mathcal{I}(\rho)- \mathcal{V}(\rho)-\mathcal{W}(\rho) dt ~:~\rho(0)=\rho^0\ ,~\rho(1)=\rho^1 \ , ~\rho(t)\in \mathcal{C}\}\ ,
\end{equation}
where $\mathcal{C}$ is the set of continuous differentiable curve in $\mathcal{P}_o(G)$. A solutions of \eqref{g_ham} is a critical point of \eqref{dnp2}.
\subsection{Complex formulations}
In this sequel, we reformulate \eqref{GNLS} into a complex wave equation. 
Let us define $$\Psi(t)=(\Psi_j(t))_{j=1}^n=(\sqrt{\rho_j(t)} e^{i \frac{ S_j(t)}{h}})_{j=1}^n\ ,$$ 
where $(\rho(t), S(t))$ are solutions of \eqref{GNLS}, then $\Psi(t)$ satisfies the following complex value ODE system.
\begin{equation}\label{CNLS}
h i \frac{d\Psi_j}{dt}=-\frac{h^2}{2}\Delta_G\Psi|_j+\Psi_j \mathbb{V}_j+\Psi_j \sum_{l=1}^n\mathbb{W}_{jl}|\Psi_l|^2\ ,
\end{equation}
in which the Laplacian on graph is defined by
\begin{equation*}
\Delta_G \Psi|_j:=-\Psi_j\big(\frac{1}{|\Psi_j|^2}\sum_{l\in N(j)}\omega_{jl}(\log \Psi_j-\log \Psi_l){g_{jl}}+\sum_{l\in N(j)}\omega_{jl}|\log \Psi_j -\log \Psi_l|^2\frac{\partial g_{jl}}{\partial \rho_j }\big)\ .
\end{equation*}
One may wonder when seeing the Laplace operator in such a nonlinear way. However, a closer examination demonstrates that this graph Laplacian is consistent with the one in the continuous case. In fact, we can show the following relationship in the continuous space. 
Let $\Psi(t,x)$ be a complex function defined in $\mathbb{R}^d$, then 
\begin{equation}\label{identity}
\Delta \Psi=\Psi\{\frac{1}{|\Psi|^2}\nabla\cdot(|\Psi|^2 \nabla \log \Psi)-|\nabla\log \Psi|^2\}\ .
\end{equation}
\begin{proof}[Proof of \eqref{identity}]
Denote $\Psi(t,x)=\sqrt{ \rho(t,x)}e^{i \frac{ S(t,x)}{h}}=e^{\frac{1}{2}\log \rho(t,x)+i \frac{ S(t,x)}{h}}$, we have
\begin{equation*}
\begin{split}
\Delta \Psi=&\nabla \cdot (\nabla \Psi)=\nabla \cdot [\Psi (\frac{1}{2}\nabla\log \rho+i \frac{\nabla  S}{h})]\\
=&\Psi[(\frac{1}{2}\nabla\log \rho+i \frac{\nabla  S}{h})^2+ (\frac{1}{2}\Delta\log \rho+i\Delta \frac{ S}{h})]\\
=&\Psi[\frac{1}{2}(\nabla \log \rho)^2+\frac{1}{2}\Delta \log \rho+i\nabla \log \rho\cdot\nabla \frac{ S}{h}+i\Delta\frac{ S}{h}-\frac{1}{4}(\nabla \log \rho)^2-(\nabla \frac{ S}{h})^2]\\
=&\Psi[\frac{1}{ \rho}\nabla\cdot( \rho \nabla(\frac{1}{2}\log \rho+ i \frac{ S}{h}))-(\frac{1}{2}\nabla\log \rho)^2-(\nabla \frac{ S}{h})^2]\\
=&\Psi(\frac{1}{|\Psi|^2}\nabla\cdot(|\Psi|^2 \nabla \log \Psi)-|\nabla\log \Psi|^2)\ ,
\end{split}
\end{equation*}
where the first equality uses $\frac{1}{ \rho}\nabla  \rho=\nabla \log \rho$. 
while the second to the last equality uses the fact\begin{equation*}
\begin{split}
&\frac{1}{ \rho}\nabla\cdot( \rho \nabla(\frac{1}{2}\log \rho+ i \frac{ S}{h}))\\
=&\frac{1}{ \rho}[\frac{1}{2}\nabla  \rho\cdot \nabla\log \rho+ \rho\Delta \log \rho +i\nabla  \rho\cdot \frac{ S}{h}+ \rho\Delta  S]\\
=&\frac{1}{2}(\nabla \log \rho)^2+\frac{1}{2}\Delta \log \rho+i\nabla \log \rho\cdot\nabla \frac{ S}{h}+i\Delta\frac{ S}{h}\ .
\end{split}
\end{equation*}
%We finish the proof.
\end{proof}

%From above computation, we propose a NLS on a graph in a nonlinear fashion, even if the interaction potential is zero. 
The nonlinearity in the Laplace operator allows the discrete NLS possessing many desirable dynamical properties, which will be shown in the next two sections.  

\section{Some Properties}\label{Ham}
%In this section, we demonstrate several properties of \eqref{GNLS}. 
For the convenience of presentation, we do not distinguish \eqref{GNLS} and its complex wave version \eqref{CNLS} in the discussion. The results here are always proposed for formulation \eqref{CNLS} while all proofs are based on $\eqref{GNLS}$. 

Our first task is examining the dispersion relation in the absence of potentials. 

\noindent\textbf{Proposition:} For a uniform toroidal graph $G$, i.e. a graph that every node has the same number of 
adjacent nodes and the weight on each edge is uniformly given,  %by $\omega_{jl}={1}/{\Delta x}$). 
the plane wave function $\Psi(t)=A (e^{i(k \cdot j\Delta x-\mu t)})_{j=1}^n$,
with any $\mu=\frac{1}{2}|k|^2$ and $A \ge 0$, satisfies 
$$i\frac{d}{dt}\Psi=-\frac{1}{2}\Delta_G\Psi\ .$$

The proposition can be verified by directly substituting the plane wave function in \eqref{GNLS}. 

In what follows, we show that \eqref{GNLS} is a well defined ODE system having several desirable properties such 
as total mass and energy conservation, time reversibility,  and gauge invariant. In addition, its interior stationary solution shares the same property as that for the counterpart in the continuous case.
\begin{theorem}\label{th5}
Given a simple weighted graph $G=(V, E, \omega)$, a vector $(\mathbb{V}_l)_{l=1}^n$, a symmetric matrix $(\mathbb{W}_{jl})_{1\leq j,l \leq n}$, and an initial condition $\Psi^0=(\Psi^0_j)_{j=1}^n$ (complex vector) satisfying 
\begin{equation*}
\sum_{j=1}^n|\Psi_j^0|^2=1\ ,\quad |\Psi_j^0|>0\ , \quad \textrm{for any $j\in V$}\ .
\end{equation*}
Then equation \eqref{GNLS} has a unique solution $\Psi(t)$ for all $t\in [0, \infty)$. Moreover, $\Psi(t)$ satisfies following properties:
\begin{itemize}
\item[(i)] It conserves the total mass $$\sum_{j=1}^n |\Psi_j(t)|^2=1\ ;$$
\item[(ii)] It conserves the total energy $$\mathcal{E}(\Psi(t))=\mathcal{E}(\Psi^0)\ ,$$
where $\mathcal{E}$ is a combination of the discrete Kinetic energy $\mathcal{E}_{kin}$, linear potential energy
$\mathcal{E}_{pot}$ and interaction potential energy $\mathcal{E}_{int}$, i.e.
\begin{equation}\label{energy}
\mathcal{E}(\Psi)=h^2\mathcal{E}_{kin}(\Psi)+\mathcal{E}_{pot}(\Psi)+\mathcal{E}_{int}(\Psi)\ .
\end{equation}
They are given by the following definitions:  
\begin{equation*}
\begin{split}
 \mathcal{E}_{kin}(\Psi)=&\frac{1}{4}\sum_{(j,l)\in E}\{[\textrm{Re}(\log \Psi_j-\log \Psi_l)]^2+[\textrm{Im}(\log \Psi_j-\log \Psi_l)]^2\}g_{jl}(|\Psi|^2)\ ,\\
 \mathcal{E}_{pot}(\Psi)=&\sum_{j=1}^n\mathbb{V}_j|\Psi_j|^2, \quad \mathcal{E}_{int}(\Psi)=\frac{1}{2}\sum_{j=1}^n\sum_{l=1}^n\mathbb{W}_{jl}|\Psi_j|^2|\Psi_l|^2\ ;
\end{split}
\end{equation*}
\end{itemize}

\item[(iii)] It is time reversible:
\begin{equation*}
\Psi(t)=\bar\Psi(-t)\ ;
\end{equation*}
\item[(iv)] It is time transverse (gauge) invariant: Denote $\Psi^\alpha(t)$ as the solution of  \eqref{GNLS} with $\mathbb{V}^\alpha=(\mathbb{V}_j+\alpha)_{j=1}^n$, 
where $\alpha$ is a given real constant,
then 
\begin{equation*}
\Psi^\alpha(t)=\Psi(t) e^{i\frac{\alpha t}{h}}\ ;
\end{equation*}
\item[(v)] The interior stationary solution of \eqref{GNLS} shares a similar property as the one in continuous state:
 If $\Psi^*(t)=\sqrt{\rho^*}e^{-i \nu t}$ satisfies \eqref{GNLS},
where $\nu\in \mathbb{R}$ and vector $\rho^*=( \rho ^*_j)_{j=1}^n\in \mathcal{P}_o(G)$ are time invariant,
then $\rho^*$ is the critical point of the minimization problem: 
\begin{equation*}
\min_{\rho\in \mathcal{P}(G)}~\mathcal{E}(\sqrt{\rho})\ ,
\end{equation*}
and
\begin{equation*}
\nu=\mathcal{E}(\sqrt{\rho^*})+\mathcal{E}_{int}(\sqrt{\rho^*})\ .
\end{equation*}
\end{theorem}
\begin{remark}
$\mathcal{E}_{kin}$ in \eqref{energy} is an analog of the Kinetic energy in continuous case: 
\begin{equation*}
\int_{\mathbb{R}^d}|\nabla \Psi|^2dx=\int_{\mathbb{R}^d}([\textrm{Re}(\nabla \log\Psi)]^2+[\textrm{Im}(\nabla \log\Psi)]^2)|\Psi|^2dx\ .
\end{equation*}
\end{remark}
\begin{remark}
Equations \eqref{GNLS} are always well defined in the interior of probability set $\mathcal{P}(G)$. In fact, we shall show that the boundary of probability set $\mathcal{P}(G)$ is a repeller for \eqref{GNLS}.
\end{remark}

\begin{proof}
We show that for any given initial condition $ \rho^0\in \mathcal{P}_o(G)$, there exists a unique solution ($ \rho(t)$, $ S(t)$) for all $t>0$. Since the right hand side of \eqref{GNLS} is locally Lipchitz continuous and $ \rho^0\in \mathcal{P}_o(G)$, from Picard's existence theorem, there exists a unique solution ($ \rho(t)$, $ S(t)$) in time interval $[0, T( \rho^0))$, where $T( \rho^0)$ is the maximal time that the solution exists. We will prove $T( \rho^0)=+\infty$ by the following claim.

\textbf{Claim 2}: {\em For any given $ \rho^0\in \mathcal{P}_o(G)$, there exists a compact set $B\subset \mathcal{P}_o(G)$, such that $T( \rho^0)=\infty$ and $ \rho(t)\in B$.}

The proof of claim 2 is based on two facts. On one hand, the ODE system \eqref{GNLS} is a Hamiltonian system on probability set, which conserves the total mass and total energy; On the other hand, the total energy contains the Fisher information $\mathcal{I}(\rho)$. On the boundary of $\mathcal{P}_o(G)$, $\mathcal{I}(\rho)$ is positive infinity, so is the total energy. From the conservation of total energy, it is not hard to see that the boundary of $\mathcal{P}_o(G)$ is a repeller for $ \rho(t)$.

\begin{proof}[Proof of Claim 2]
We construct a set $B\subset \mathcal{P}(G)$:
\begin{equation*}
B=\{ \rho\in \mathcal{P}(G)~:~\frac{h^2}{8}\mathcal{I}(\rho)\leq  \mathcal{E}(\Psi^0)-\min_{ \rho\in \mathcal{P}(G)}[\mathcal{V}( \rho)+\mathcal{W}( \rho)]~\}\ ,
\end{equation*}
where $\mathcal{E}(\Psi^0)=\mathcal{H}( \rho^0,  S^0)=\frac{1}{2}(\nabla_G S^0, \nabla_G S^0)_{ \rho^0}+\frac{h^2}{8}\mathcal{I}( \rho^0)+\mathcal{V}( \rho^0)+\mathcal{W}( \rho^0)<\infty\ .$ Obviously, $B$ is not empty.

We will prove that $B$ is a compact set and $\rho(t)\subset B$ for all $t>0$ by following three steps.

Step 1, we prove (i) and (ii) for $t\in [0, T( \rho^0))$. Since 
\begin{equation*}
\sum_{j=1}^n\frac{d\rho_j}{dt}=-\sum_{j=1}^n \textrm{div}_G( \rho \nabla_G  S)|_j =0\ ,
\end{equation*}
(i) is concluded.  For (ii), we need to show \begin{equation*}
\frac{d}{dt}\mathcal{E}(\Psi(t))=0\ ,
\end{equation*}
where $\mathcal{E}(\Psi)=\mathcal{H}(\rho, S)$.
Notice \eqref{GNLS} has the following symplectic form
\begin{equation*}
\frac{d}{dt}\begin{pmatrix}   \rho \\  S \end{pmatrix}=\mathbb{J}\begin{pmatrix}\frac{\partial}{\partial  \rho}\mathcal{H}\\ \frac{\partial}{\partial  S}\mathcal{H}\end{pmatrix}\ ,
\end{equation*}
then
 \begin{equation*}
 \begin{split}
 \frac{d}{dt}\mathcal{E}(\Psi(t))=&\frac{d}{dt}\mathcal{H}( \rho(t), S(t))=\sum_{j=1}^n\{\frac{\partial}{\partial \rho_j}\mathcal{H} \frac{d}{dt}\rho_j+\frac{\partial}{\partial  S_j}\mathcal{H} \frac{d}{dt} S_j\}\\
 =&\sum_{j=1}^n\{\frac{\partial}{\partial \rho_j}\mathcal{H} \frac{\partial }{\partial  S_j}\mathcal{H}-\frac{\partial}{\partial  S_j}\mathcal{H}\frac{\partial}{\partial \rho_j}\mathcal{H}\}=0\ .
\end{split}
\end{equation*}

Step 2, we show that $\mathcal{I}(\rho)$ is positive infinity on the boundary, i.e.
\begin{equation*}
\lim_{\min_{j\in V}{\rho_j}\rightarrow 0}\mathcal{I}(\rho)=+\infty\ .
\end{equation*}
Assume the above is not true, there exists a constant $M>0$, such that if $\min_{i\in V}\rho_j=0$, then
\begin{equation*}
M\geq \mathcal{I}(\rho)=\frac{1}{2}\sum_{(j,l)\in E}\omega_{jl}(\log\rho_j-\log\rho_l)^2\frac{\rho_j+\rho_l}{2}\geq \frac{1}{4}\sum_{(j,l)\in E}\omega_{jl}(\log\rho_j-\log\rho_l)^2 \max\{\rho_j, \rho_l\}\ .
 \end{equation*} 
Hence for any $(j,l)\in E$, we have $$\omega_{jl}(\log\rho_j-\log\rho_l)^2 \max\{\rho_j, \rho_l\}\leq 2M<+\infty\ .$$ 
Since there exists a $j^*\in V$, such that $ \rho_{j^*}=0$, the above formula implies that for any $l\in N(j^*)$, $\rho_l=0$. 
%Now we look at vertices $k\in N(j)$. For the same reason, we require $ \rho_k=0$. 
Since $G$ is connected and $V$ is a finite set, by iterating through the nodes, we get $ \rho_1=\cdots= \rho_n=0$, which contradicts the fact that $\sum_{j=1}^n\rho_j=1$. 

Step 3, we claim that $B$ is a compact set. This can be easily verified because $\mathcal{I}$ is a lower semi continuous function, and $\mathcal{I}(\rho)=+\infty$ when $\rho\in \mathcal{P}(G)\setminus \mathcal{P}_o(G)$. 
Hence $B$ is a compact set in $\mathbb{R}^n$.

Let us combine above three steps. Since \eqref{GNLS} is a Hamiltonian system in $\mathcal{P}_o(G)$,
\begin{equation*}
\mathcal{E}(\Psi(t))=\mathcal{E}(\Psi^0)=\frac{1}{2}(\nabla_G S(t), \nabla_G S(t))_{ \rho(t)}+\frac{h^2}{8}\mathcal{I}( \rho(t))+\mathcal{V}( \rho(t))+\mathcal{W}( \rho(t))\ ,
\end{equation*}
then
\begin{equation*}
\begin{split}
\frac{h^2}{8}\mathcal{I}( \rho(t))=& \mathcal{E}(\Psi^0)-\frac{1}{2}(\nabla_G S(t), \nabla_G S(t))_{ \rho(t)}-\big(\mathcal{V}( \rho(t))+\mathcal{W}( \rho(t))\big)\\
\leq&\mathcal{E}(\Psi^0)-\min_{ \rho\in \mathcal{P}(G)}[\mathcal{V}( \rho)+\mathcal{W}( \rho)]\ .\\
\end{split}
\end{equation*}
Thus $ \rho(t)\in B\subset \mathcal{P}_o(G)$ for all $t>0$.

\end{proof}
Next, we prove (iii) and (iv). For (iii), since $\Psi_j=\sqrt{\rho_j} e^{i \frac{ S_j}{h}}$, its conjugate $\bar \Psi$ satisfies 
$$\bar\Psi_j=\sqrt{\rho_j} e^{i \frac{\bar  S_j}{h}}\quad \textrm{with $\bar S_j=- S_j$\ .}$$  Let us look at \eqref{GNLS} by changing $t$ to $-t$.
\begin{equation*}
\left \{
\begin{aligned}
&-\frac{d\rho_j}{dt}+\sum_{l\in N(j)}\omega_{lj}( S_l- S_j)g_{lj}(\rho)=0\ ;\\
&-\frac{d S_j}{dt}+\frac{1}{2}\sum_{j\in N(i)}\omega_{jl}( S_j- S_l)^2\frac{\partial g_{jl}}{\partial \rho_j}+\frac{\partial }{\partial \rho_j}\{\frac{h^2}{8}\mathcal{I}(\rho)+\mathcal{W}( \rho)+\mathcal{V}(\rho)\}=0\ .
\end{aligned}
\right.
\end{equation*}
Denote $\bar S=- S$, then $( \rho(t),  S(t))$ and $( \rho(-t), \bar S(-t))$ satisfies \eqref{GNLS}. 

For (iv), if $V\rightarrow V^\alpha=V+\alpha$, we substitute $\Psi^\alpha(t)=\Psi(t) e^{i \frac{\alpha t}{h}}$ into \eqref{GNLS} to get:
\begin{equation*}
\left \{
\begin{aligned}
&\frac{d\rho_j}{dt}+\sum_{l\in N(j)}\omega_{jl}[( S_l+\alpha t)-( S_j+\alpha t)]g_{jl}(\rho)=0\ ;\\
&\frac{d}{dt}( S_j+\alpha t)+\frac{1}{2}\sum_{l\in N(j)}\omega_{jl}[( S_l+\alpha t)-( S_j+\alpha t)]^2\frac{\partial g_{jl}}{\partial \rho_j}+\mathbb{V}_j+\frac{\partial }{\partial \rho_j}\{\frac{h^2}{8}\mathcal{I}(\rho)+\mathcal{W}( \rho)\}=0\ .
\end{aligned}
\right.
\end{equation*}
This means that if $(\rho, S)$ are solutions of \eqref{GNLS} with $V$, then 
$ S^\alpha(t)= S+\alpha t$, $ \rho^\alpha(t)= \rho(t) $ are solutions of \eqref{GNLS} with $V^\alpha$, i.e.
\begin{equation*}
\Psi^\alpha=\sqrt{ \rho^\alpha}e^{i\frac{ S^\alpha}{h}}=\sqrt{\rho}e^{i\frac{ S}{h}}e^{i\frac{\alpha t}{h}}=\Psi e^{i\frac{\alpha t}{h}}\ .
\end{equation*}

(v). Substituting the stationary solution $\Psi^*$ into \eqref{GNLS}, we observe
\begin{equation*}
\nu=\frac{\partial }{\partial\rho_j}\big(\frac{h^2}{8}\mathcal{I}(\rho)+\mathcal{V}(\rho)+\mathcal{W}(\rho)\big)|_{\rho=\rho^*}\ ,\quad \textrm{for any $i\in V$}\ .
\end{equation*}
Notice that $\mathcal{E}(\sqrt{\rho})=\frac{h^2}{8}\mathcal{I}(\rho)+\mathcal{V}(\rho)+\mathcal{W}(\rho)$. It is simple to check that $\rho^*$ satisfies the Karush-Kuhn-Tucker conditions of minimization 
 \begin{equation*}
 \min_{\rho}\{\mathcal{E}(\sqrt{\rho}):~\sum_{j=1}^n\rho_j=1,\quad\rho_j> 0\}\ ,
\end{equation*}
with $\nu$ being the Lagrange multiplier. Next, we show 
\begin{equation*}
\begin{split}
\nu=&\sum_{j=1}^n\nu\rho_j^*=\sum_{j=1}^n\frac{\partial}{\partial\rho_j}\{\frac{h^2}{8}\mathcal{I}+\mathcal{V}+\mathcal{W}\}|_{\rho^*}\cdot\rho_j^*\\
=&\frac{h^2}{8}\sum_{j=1}^n \frac{\partial}{\partial\rho_j}\mathcal{I}|_{\rho^*}\rho^*_j+\sum_{j=1}^n [\mathbb{V}_j+\sum_{l=1}^n\mathbb{W}_{jl}\rho^*_l]\rho^*_j\\
=&\frac{h^2}{8}\sum_{j=1}^n \frac{\partial}{\partial\rho_j}\mathcal{I}|_{\rho^*}\rho_j^*-\frac{h^2}{8}\mathcal{I}( \rho ^*)+\frac{h^2}{8}\mathcal{I}( \rho ^*)+\mathcal{V}( \rho ^*)+\mathcal{W}( \rho ^*)+\mathcal{W}( \rho ^*)\\
=&\frac{h^2}{8}\big(\sum_{j=1}^n \frac{\partial}{\partial\rho_j}\mathcal{I}|_{\rho^*}\rho_j^*-\mathcal{I}( \rho ^*)\big)+\mathcal{E}(\sqrt{\rho^*})+\mathcal{W}( \rho ^*)\\
=&\mathcal{E}(\sqrt{\rho^*})+\mathcal{E}_{int}(\sqrt{\rho^*})\ ,
\end{split}
\end{equation*}
where the last equality is from the fact: $\mathcal{E}_{int}(\sqrt{\rho})=\mathcal{W}(\rho)$ and $\mathcal{I}(\rho)=\sum_{j=1}^n\frac{\partial}{\partial \rho_j}\mathcal{I}(\rho)\rho_j$.
\end{proof}

\subsection{Ground states}\label{stability}
Similar to Nelson's idea in \cite{Nelson1}, we show that the stationary solution of \eqref{GNLS} in Theorem \ref{th5} (v) is related to the discrete ground state of the NLS.
\begin{corollary}\label{col}
If $\mathbb{W}$ is a semi positive definite matrix, then the stationary state is a ground state
$\Psi^g=\sqrt{\rho^g}e^{-i \nu^g t}$, i.e.
\begin{equation}\label{g_s}
\Psi^g=\arg\min_{\Psi}\{~\mathcal{E}(\Psi)~:~\sum_{j=1}^n|\Psi_j|^2=1\}\ ,
\end{equation}
with 
\begin{equation*}
\rho^g=\arg\min_{\rho\in \mathcal{P}(G)}~\mathcal{E}(\sqrt{\rho})
\quad \textrm{and}\quad\nu^g=\mathcal{E}(\sqrt{\rho^g})+\mathcal{E}_{int}(\sqrt{\rho^g})\ .
\end{equation*}
\end{corollary}
\begin{proof}
From Theorem 5 (v), $\rho^g$ is a critical point of $\min_{\rho\in\mathcal{P}_o(G)}~\mathcal{E}(\sqrt{\rho})$ and $\nu^g$ is defined as above. We only need to prove $\Psi^g$ is the minimizer of problem \eqref{g_s}. In fact, \begin{equation*}
\min_{\Psi}\{~\mathcal{E}(\Psi)~:~\sum_{j=1}^n|\Psi_j|^2=1\}\geq \min_{\rho\in \mathcal{P}(G)}~\mathcal{E}(\sqrt{\rho})\ .\end{equation*}
because
\begin{equation*}
\begin{split}
\mathcal{E}(\Psi)=&\frac{1}{4}\sum_{(j,l)\in E}\omega_{jl}(S_j-S_l)^2g_{jl}(\rho)+\frac{h^2}{8}\mathcal{I}(\rho)+\mathcal{V}(\rho)+\mathcal{W}(\rho)\\
\geq &\frac{h^2}{8}\mathcal{I}(\rho)+\mathcal{V}(\rho)+\mathcal{W}(\rho)=\mathcal{E}(\sqrt{\rho})\ ,
\end{split}
\end{equation*}
the equality holds if and only if $S_j=S_l$, for any $(j,l)\in E$. Since $G$ is a connected graph, then a ground state $\Psi^g=\sqrt{\rho^g}e^{{i}\frac{S^g}{h}}$ has the following structure: 
$$\rho^g=\arg\min_{\rho\in\mathcal{P}(G)}\mathcal{E}(\sqrt{\rho}) \quad \textrm{and} \quad \textrm{$S_1^g=S_2^g=\cdots =S_n^g$\ .} $$

Next, we show that the function $\mathcal{E}(\sqrt{\rho})=\frac{h^2}{8}\mathcal{I}(\rho)+\mathcal{W}(\rho)+\mathcal{V}(\rho)$ is strictly convex.
If this is true, we can conclude that $\rho^g$ is a unique minimizer, which is the ground state.

Notice that $\mathcal{W}(\rho)=\frac{1}{2}\sum_{j=1}^n\sum_{l=1}^n\mathbb{W}_{jl}\rho_j\rho_l$, $\mathcal{V}(\rho)=\sum_{j=1}^n \mathbb{V}_j\rho_j$ are convex functionals. So we only need to prove 
\begin{center}
{\em $\mathcal{I}(\rho)$ is a strict convex functional in $\mathcal{P}_o(G)$\ .}
\end{center}
We show this result by proving
\begin{equation}\label{claim3}
\min_{\sigma\in T_\rho\mathcal{P}_o(G)}\{\sigma^T \textrm{Hess}_{\mathbb{R}^n}\mathcal{I}(\rho) \sigma~:~\sigma^T\sigma=1\}>0\ .
\end{equation}
Since the Hessian matrix of $\mathcal{I}$ is
\begin{equation*}
\frac{\partial^2}{\partial\rho_l\partial\rho_j}\mathcal{I}(\rho)=
\begin{cases}
-\frac{1}{\rho_l\rho_j}\omega_{lj}t_{lj}&\textrm{if $l\in N(j)$}\ ;\\
\frac{1}{\rho_j^2}\sum_{l\in N(j)}\omega_{lj}t_{lj} &\textrm{if $l=j$}\ ;\\
0& \textrm{otherwise}\ ,\\
\end{cases}
\end{equation*}
where 
\begin{equation}\label{tij}
t_{lj}=(\rho_l-\rho_j)(\log\rho_l-\log\rho_j)+(\rho_l+\rho_j)>0\ ,
\end{equation}
hence
\begin{equation*}
\begin{split}
\sigma^T \textrm{Hess}_{\mathbb{R}^n}\mathcal{I}(\rho) \sigma
=&\frac{1}{2}\sum_{(l,j)\in E}  t_{lj}\{(\frac{\sigma_j}{\rho_j})^2+ (\frac{\sigma_l}{\rho_l})^2  -2 \frac{\sigma_l}{\rho_l} \frac{\sigma_j}{\rho_j}\}\\
=&\frac{1}{2}\sum_{(l,j)\in E} t_{lj}(\frac{\sigma_j}{\rho_j}-\frac{\sigma_l}{\rho_l})^2\geq 0\ .
\end{split}
\end{equation*}
So $\textrm{Hess}_{\mathbb{R}^n}\mathcal{I}$ is a semi-positive definite matrix. 

Suppose \eqref{claim3} is not true, there exists a unit vector $\sigma^*\in T_\rho\mathcal{P}_o(G)$, such that 
\begin{equation*}
\sigma^{*T} \textrm{Hess}_{\mathbb{R}^n}\mathcal{I}(\rho) \sigma^*=\frac{1}{2}\sum_{(l,j)\in E} t_{lj}(\frac{\sigma_l^*}{\rho_l}-\frac{\sigma^*_j}{\rho_j})^2= 0\ .
\end{equation*}
Then $\frac{\sigma_1^*}{\rho_1}=\frac{\sigma_2^*}{\rho_2}=\cdots \frac{\sigma_n^*}{\rho_n}=0$.
Combining with $\sum_{j=1}^n\sigma_j^*=0$, we have $\sigma_1^*=\sigma_2^*=\cdots=\sigma_n^*=0$, which contradicts that $\sigma^*$ is a unit vector.
\end{proof}
%The above theorems suggest us to consider the following equation: 
It is worth mentioning that  we have the following eigenvalue problem at the ground state:
\begin{equation}\label{ground}
\nu \Psi_j= -\frac{h^2}{2}\Delta_G \Psi_j+ \mathbb{V}_j\Psi_j+\Psi_j\sum_{l=1}^n \mathbb{W}_{jl}|\Psi_l|^2\ ,\quad (\Psi_j)_{j=1}^n\in \mathbb{R}^n\ .
\end{equation}
%where $\Delta_G$ represents the associated nonlinear Laplacian. 
The solution of \eqref{ground} is the ground state configuration, where $|\Psi|=\sqrt{\rho^g}$ and $\nu=\mathcal{E}(\sqrt{\rho^g})+\mathcal{E}_{int}(\sqrt{\rho^g})$ is the associated energy level. 

\section{Linearized problems}
In this section, 
%we show that the formulation of ground state \eqref{g_s} and NLS \eqref{GNLS} not only share similar structures in continuous space \cite{Bao1}, but also brings many interesting dynamical properties. For instance, 
we study the linearized problem near the ground state. 
Consider the Hamiltonian system \begin{equation*}
\frac{d}{dt}\begin{pmatrix}   \rho \\  S \end{pmatrix}=\mathbb{J}\begin{pmatrix}\frac{\partial}{\partial  \rho}\mathcal{H}\\ \frac{\partial}{\partial  S}\mathcal{H}\end{pmatrix}\ .
\end{equation*}
The ground state $(\rho^g, S^g(t))$ can be viewed as its equilibrium solution, which is clearly the critical point of Hamiltonian $$\mathcal{H}(\rho, S)=\frac{1}{2}(\nabla_G S, \nabla_G S)_{\rho}+\frac{h^2}{8}\mathcal{I}(\rho)+\frac{1}{2}\rho^T\mathbb{W}\rho+\mathbb{V}^T \rho\ .$$ 
Consider the linearized problem of \eqref{GNLS} 
\begin{equation*}
\frac{d}{dt}z= H^{(2)} z\ ,
\end{equation*}
where $z\in \mathbb{R}^{2n}$,  $H^{(2)}\in \mathbb{R}^{2n\times 2n}$ is the {\em Hamiltonian matrix} at the equilibrium $( \rho ^g, S^g)$.

%By direct calculations, we find that $D^2_{\rho S}\mathcal{H}|_{( \rho ^g, S^g)}=0$, 
%$D^2_{\rho \rho}\mathcal{H}|_{( \rho ^g, S^g)}=\mathbb{W}+\frac{h^2}{8}L_1( \rho ^g)$ ,
%$D^2_{SS}\mathcal{H}|_{( \rho ^g, S^g)}=L_2( \rho ^g)$, where $D^2$ is the second differential operator, and 
%and $L_m( \rho ^g)=(L_m( \rho ^g)(j,l))_{1\leq j,l\leq n}\in \mathbb{R}^{n\times n}$, $m=1,2$ are defined by: 
% $$L_1( \rho ^g)(j,l):=\begin{cases}
%-\frac{1}{\rho_l^g\rho_j^g}\omega_{lj}t_{lj}( \rho ^g)&\textrm{if $l\in N(j)$}\\
%\frac{1}{(\rho_j^{g})^2}\sum_{l\in N(j)}\omega_{lj}t_{lj}( \rho ^g) &\textrm{if $l=j$}\\
%0& \textrm{otherwise}\\
%\end{cases}\ ,$$ $t_{lj}$ is defined in \eqref{tij} and $$L_2( \rho ^g)(j,l):=\begin{cases}
%-\omega_{lj}g_{lj}(\rho ^g)&\textrm{if $l\in N(j)$}\\
%\sum_{l\in N(j)}\omega_{lj}g_{lj}(\rho ^g) &\textrm{if $l=j$}\\
%0& \textrm{otherwise}\ \\
%\end{cases}\ .$$
Because $(S^g_j)_{j=1}^n$ is a constant vector, we obtain a simple structure for $H^{(2)}$: 
\begin{equation}\label{Hessian}
H^{(2)}:=\mathbb{J}\cdot \textrm{Hess}_{\mathbb{R}^{2n}}\mathcal{H}( \rho , S)|_{( \rho ^g, S^g)}=\begin{pmatrix}
0&  L_2( \rho ^g)\\
 -\mathbb{W}-\frac{h^2}{8}\textrm{Hess}_{\mathbb{R}^n}\mathcal{I}(\rho^g) & 0
\end{pmatrix}\ .
\end{equation}
%Here \eqref{Hessian} is derived by a direct calculation. 
%A closer look will find that $L_1$, $L_2$ are indeed the Hessian matrices of the Fisher information and kinetic energy, $\frac{1}{2}(\nabla_G S,\nabla_G S)_{\rho}$, respectively. 
%%%\begin{remark} It is worth mentioning that \eqref{Hessian} has a continuous analogs 
%%%\begin{equation*}
%%%\mathbb{J}\mathcal{H}|_{(\rho^g, S^g)}=\begin{pmatrix}
%%%0& L(\rho^g)  \\
%%%-\mathbb{W}(x,y)-\frac{h^2}{8}\textrm{Hess}_{L^2}\mathcal{I}  (\rho^g)& 0
%%%\end{pmatrix}\ ,
%%%\end{equation*}
%%%where $L_1(\rho^g)$, $L_2(\rho^g)$ are $L^2$ second variation of Fisher information and kinetic energy ($\int_{\mathbb{R}^d}(\nabla S(x))^2\rho(x)dx$) at the ground state $\rho^g$. In details, let $\delta \rho(x)$, $\delta S(x)\in C^{\infty}(\mathbb{R}^d)$, then 
%%%\begin{equation*}
%%%L_1(\rho)(\delta \rho)=-\frac{1}{\rho}\nabla\cdot (\rho \nabla \frac{\delta\rho}{\rho} )\ ,
%%% \quad\textrm{and}\quad L_2(\rho)(\delta S)=-\nabla\cdot (\rho \nabla \delta S)\ .
%%%\end{equation*}
%%%\end{remark}
We estimate the eigenvalue of \eqref{Hessian} in a particular case: 
\begin{equation}\label{BEC}
h i \frac{d\Psi_j}{dt}=-\frac{h^2}{2}\Delta_G\Psi|_j+\alpha \Psi_j |\Psi_j|^2\ .
\end{equation} 
This equation is obtained from \eqref{CNLS} by taking $\mathbb{V} = 0$ and $\mathbb{W}=\alpha \mathbb{I}$. It can be viewed as a discrete version of Gross-Pitaevskii equation (GPE), which has been proposed to model the Bose-Einstein condensate. 
%\begin{theorem}

\noindent\textbf{Proposition:} For discrete GPE \eqref{BEC}, $H^{(2)}$ has eigenvalues  $$\alpha_k^+:=+i\sqrt{\frac{1}{4}\lambda_k^2h^2+\frac{\alpha \lambda_k}{n}} \ , \quad\alpha_k^-:=-i\sqrt{\frac{1}{4}\lambda_k^2h^2+\frac{\alpha \lambda_k}{n}}\ , $$ with associated eigenvectors 
 $$ w_k^+:=\begin{pmatrix} v_k\\  i\sqrt{\frac{n^2h^2}{4}+\frac{\alpha n}{\lambda_k}} v_k\end{pmatrix}\ ,\quad w_k^-:=\begin{pmatrix} {i} v_k\\ 
 \sqrt{\frac{n^2h^2}{4}+\frac{\alpha n}{\lambda_k}} v_k\end{pmatrix}\in \mathbb{C}^{2n}\ ,
$$
i.e. 
$$H^{(2)}w_k^+=\alpha_k^+w_k^+\ ,\quad H^{(2)}w_k^-=\alpha_k^-w_k^-\ .$$
Here $Lv_k=\lambda_k v_k$. $\lambda_k\geq 0$, $v_k\in \mathbb{R}^n$ are $k$-th eigenvalue and eigenvector of graph Laplacian matrix $L=(L(j,l))_{1\leq j,l\leq n}\in \mathbb{R}^{n\times n}$, where 
$L(j,l)=\begin{cases}
-\omega_{jl}&\textrm{if $l\in N(j)$\ ;}\\
\sum_{l\in N(j)}\omega_{jl} &\textrm{if $l=j$\ ;}\\
0& \textrm{otherwise\ .}\ \\
\end{cases}$
%\end{theorem}
\begin{proof}
Denote $\textbf{1}=(\frac{1}{n})_{j=1}^n$. From Karush-Kuhn-Tucker conditions, one can easily find $\rho^g=\textbf{1}$ is the critical point of  $\{\frac{h^2}{8}\mathcal{I}(\rho)+\frac{\alpha}{2} \sum_{j=1}^n\rho_j^2~:~\rho\in \mathcal{P}(G)\}$. In this case, $L_1(\textbf{1})=2n L$, $L_2(\textbf{1})=\frac{1}{n}L$. So the matrix \eqref{Hessian} becomes
\begin{equation}\label{Hessian1}
H^{(2)}=\mathbb{J}\cdot \textrm{Hess}_{\mathbb{R}^{2n}}\mathcal{H}( \rho , S)|_{(\textbf{1}, S^g)}
=\begin{pmatrix}
0&  \frac{1}{n}L \\
-\alpha \mathbb{I}-\frac{nh^2}{4}L& 0
\end{pmatrix}\ .
\end{equation}
In fact, we can find all eigenvalues and eigenvectors of \eqref{Hessian}.
Notice that $L\in \mathbb{R}^{n\times n}$ is a semi-positive matrix, and denote $\lambda_k\in \mathbb{R}$, $v_k\in \mathbb{R}^n$, as the $k$-th eigenvalue and eigenvector of $L$. Therefore one can check that 
 \begin{equation*}
 \begin{split}
H^{(2)}w_k^+=& \begin{pmatrix}
0&  \frac{1}{n}L \\
-\alpha \mathbb{I}-\frac{nh^2}{4}L& 0
\end{pmatrix}  \begin{pmatrix} v_k\\  i\sqrt{\frac{n^2h^2}{4}+\frac{\alpha n}{\lambda_k}} v_k\end{pmatrix}  \\
=& \begin{pmatrix}  i \frac{1}{n} \sqrt{\frac{n^2h^2}{4}+\frac{\alpha n}{\lambda_k}}\cdot \lambda_k v_k\\  -\alpha v_k-\frac{nh^2}{4}\lambda_k v_k\end{pmatrix}  \\
=&+i\sqrt{\frac{\lambda^2_kh^2}{4}+\frac{\alpha \lambda_k}{n}} \begin{pmatrix} v_k\\  i\sqrt{\frac{n^2h^2}{4}+\frac{\alpha n}{\lambda_k}} v_k\end{pmatrix}\\
=& \alpha_k^+ w_k^+\\
\end{split}
\end{equation*}
Similarly, $H^{(2)}w_k^-=\alpha_k^-w_k^-$.
\end{proof}
From these eigenvalues, when $\alpha> -\frac{n}{4}\lambda_k h^2$, the solution $\rho=\textbf{1}$, $\nabla_GS=0$ is stable for \eqref{BEC}. 
When $\alpha= -\frac{n}{4}\lambda_k h^2$, bifurcations may happen. 

\section{Examples}
Finally, we demonstrate \eqref{GNLS} and \eqref{ground} by two numerical examples. 
\begin{example}[NLS on a two points graph]
Consider a Hamiltonian:
\begin{equation*}
\mathcal{H}( \rho , S)=\frac{1}{2}(S_1-S_2)^2g_{12}(\rho)+\frac{h^2}{8}(\log \rho_1-\log \rho_2)^2g_{12}(\rho)+\mathbb{V}_1\rho_1+\mathbb{V}_2\rho_2\ ,
\end{equation*}
where $\mathbb{V}_1=\mathbb{V}_2=c$. In this case, the solution of \eqref{GNLS}, $( \rho _1(t), \rho_2(t), S_1(t)-S_2(t))\in \mathbb{R}^3$, can be plotted using a phase portrait.
\begin{figure}
{\includegraphics[scale=0.3]{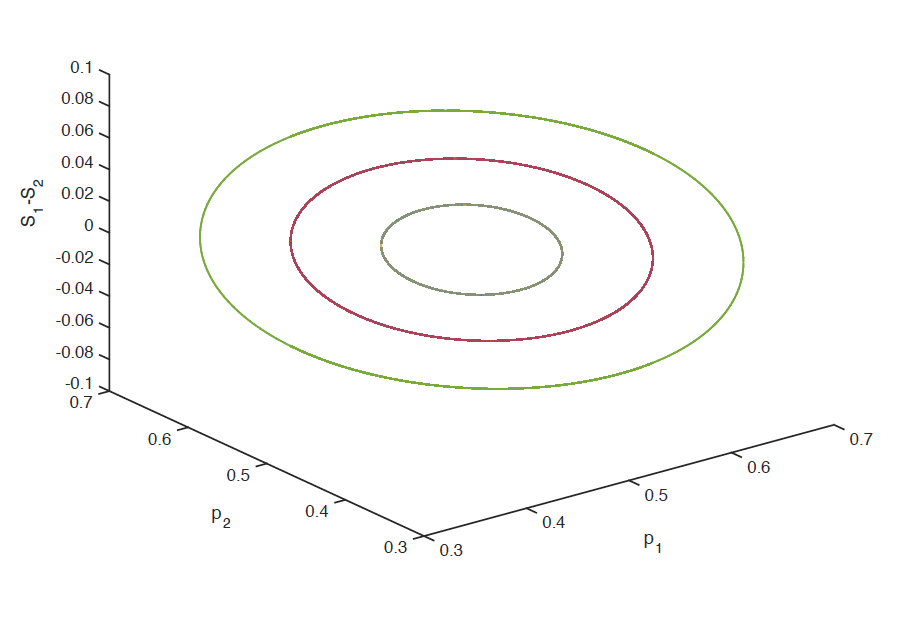}}
\caption{The phase portrait of $( \rho _1(t), \rho_2(t), S_1(t)-S_2(t))$ with different initial conditions.}
\label{Fig}
\end{figure}
In Figure \ref{Fig}, each circle represents a trajectory of \eqref{GNLS}. The ground state $(\frac{1}{2}, \frac{1}{2},0)$ is in the center of all these circles, so it is spectrally stable. 
%This experiment matches Corollary \ref{col7}.
\end{example}

\begin{example}[Ground state]
We demonstrate the ground states on a 1-D lattice graph. Set $\mathbb{W}=0$. Consider the following minimization problem 
$$\rho^g=\arg\min_{\rho\in \mathcal{P}(G)}\sum_{j=1}^n \mathbb{V}_j\rho_j+\frac{h^2}{8}\mathcal{I}(\rho)\ ,$$
where $\mathbb{V}_j=\frac{x_j^2}{2}$. We compute the above minimizer numerically \cite{li-thesis} in the interval $[-5,5]$ with $n=20$. From Figure \ref{fig3}, we observe that the ground state approaches to the delta measure supported at $0$ when $h \rightarrow 0$. This captures the exactly same effect in continuous states, in which the ground state is a Gaussian distribution with variance $h^2$ \cite{Bao1}.
\begin{figure}
\centering
{\includegraphics[scale=0.3]{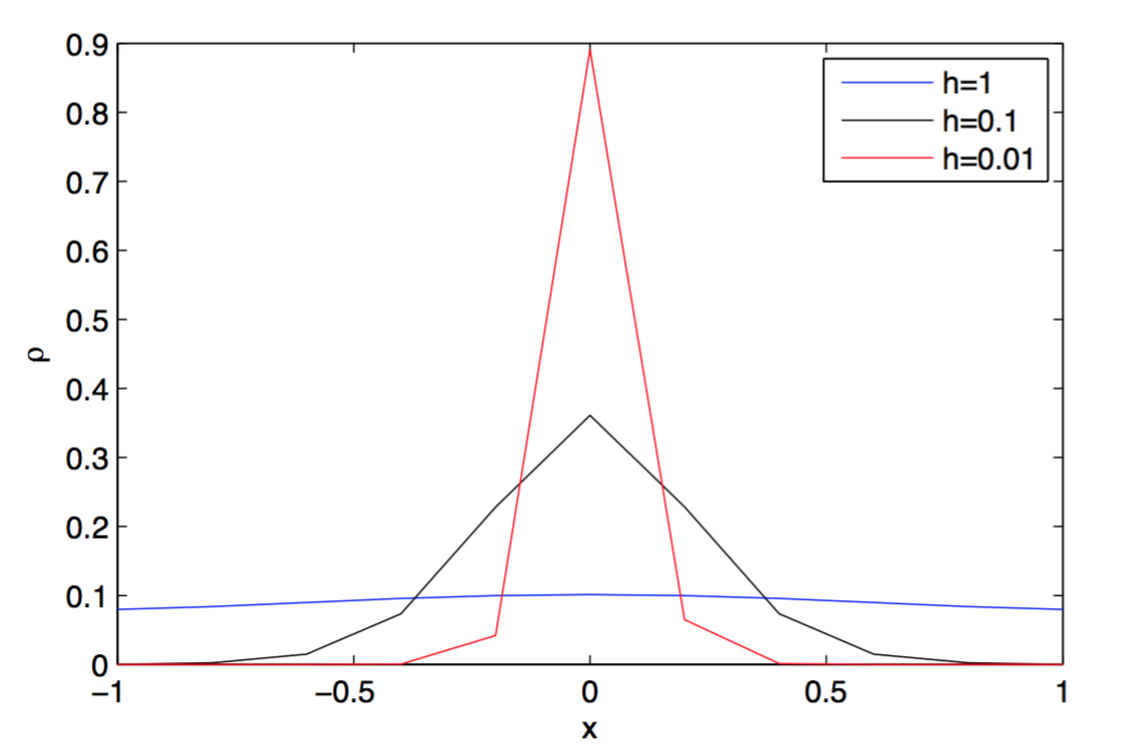}}
\caption{The plot of ground state $\rho^g$. The blue, black, red curves represents $h=1$, $0.1$, $0.01$, respectively.}
\label{fig3}
\end{figure}
\end{example}

\section{Conclusions}
In this paper we have introduced a new NLS on finite graphs \eqref{GNLS}. Compared to the existing work, \eqref{GNLS} has the following distinct features: First, the discrete NLS is introduced via discrete optimal transport. This formulation provides a way to study the discrete NLS from geometric viewpoint; Second, the discrete Fisher information $\mathcal{I}(\rho)$ is applied to construct Hamiltonian system. Because of it, \eqref{GNLS} conserves total energy and matches the stationary solution. Last but not the least, it introduces the ground state on graph by \eqref{ground}. Studying the stability problem around the discrete ground state introduces a Hamiltonian matrix, which is a symplectic composition of two modified graph Laplacian matrices. These give insights of the system that can be explored in the future.

\textbf{Acknowledgement}: We would like to thank Prof. Eric Carlen, Prof. Paul Goldbart and Dr. Jun Lu for fruitful discussions.

\end{document}